\documentclass{amsart}
\usepackage{geometry}
\usepackage{graphicx}
\usepackage[dvipsnames,usenames]{xcolor}
\usepackage[hidelinks,colorlinks=true,linkcolor=blue, citecolor=black,linktocpage=true]{hyperref}
\usepackage{mathrsfs}
\usepackage{mathtools}
\usepackage[UKenglish]{babel}
\usepackage{amssymb,amsxtra,dsfont}
\usepackage[noabbrev]{cleveref}
\usepackage{enumerate}
\usepackage{tikz}
\usepackage{comment}
\usetikzlibrary{trees,decorations.pathmorphing,decorations.markings,matrix,shapes}

\makeatother

\newtheorem{theorem}{Theorem}[section]
\newtheorem*{theorem*}{Theorem}

\newtheorem{proposition}[theorem]{Proposition}

\newtheorem{corollary}[theorem]{Corollary}
\newtheorem{claim}[theorem]{Claim}
\newtheorem{lemma}[theorem]{Lemma}
\newtheorem{fact}[theorem]{Fact}

\theoremstyle{definition}
\newtheorem{definition}[theorem]{Definition}
\newtheorem{question}[theorem]{Question}

\newtheorem{example}[theorem]{Example}
\newtheorem{remark}[theorem]{Remark}

\numberwithin{equation}{section}

\def\lk{\text{lk}}

\newcommand{\R}{\mathbb R}
\newcommand{\x}{\times}

\newcommand{\Q}{\mathbb{Q}}
\newcommand{\N}{\mathbb{N}}
\newcommand{\Z}{\mathbb{Z}}

\newcommand{\SFH}{\text{SFH}}
\newcommand{\HF}{\text{HF}}
\newcommand{\HFK}{\widehat{\text{HFK}}}
\newcommand{\HFL}{\widehat{\text{HFL}}}
\newcommand{\fs}{\mathfrak{s}}
\newcommand{\ft}{\mathfrak{t}}

\begin{document}

\title{New Invariants for Virtual Knots via Spanning Surfaces}

\author{Andr\'as Juh\'asz}

\address{Mathematical Institute, 
	University of Oxford, 
	Andrew Wiles Building, 
	Radcliffe Observatory Quarter, 
	Woodstock Road, 
	Oxford, 
	OX2 6GG,
	UK}

\email{juhasza@maths.ox.ac.uk}

\author{Louis H. Kauffman}

\address{Department of Mathematics, Statistics and Computer Science,
University of Illinois at Chicago,
851 South Morgan Street,
Chicago, Illinois 60607-7045, 
USA}

\email{kauffman@uic.edu}

\author{Eiji Ogasa}  

\address{Meijigakuin University, Computer Science\\ 
Yokohama, Kanagawa, 244-8539 \\
Japan} 

\email{pqr100pqr100@yahoo.co.jp, ogasa@mail1.meijigkakuin.ac.jp} 

\begin{abstract}
	We define three different types of spanning surfaces for knots in thickened surfaces. We use these to introduce
	new Seifert matrices, Alexander-type polynomials, genera, and a signature invariant. One of these Alexander polynomials extends to 
	virtual knots and can obstruct a virtual knot from being classical. Furthermore, it can distinguish a knot in a thickened surface from its mirror up to isotopy. We also propose several constructions
	of Heegaard Floer homology for knots in thickened surfaces, and give examples why they are not stabilization invariant. 
	However, we can define Floer homology for virtual knots by taking a minimal genus representative.  
	Finally, we use the Behrens--Golla $\delta$-invariant to obstruct a knot from being a stabilization of another.
\end{abstract}

\maketitle

\section{Introduction }\label{secMain}

Virtual knots were introduced by the second author \cite{Kauffman, Kauffmani}. These are 
knots in thickened surfaces up to a stabilization operation; see Section~\ref{secVlink}. 
In this paper, we define three different types of spanning surfaces for knots in thickened surfaces that we call twin, web, and bracelet Seifert surfaces, respectively, and two genus-type numerical invariants $\varepsilon$ and $\xi$.
We use these surfaces to define new Alexander-type polynomials called the \emph{twin Alexander polynomial} (derived from a new type of Seifert matrix called the \emph{twin Seifert matrix}), the \emph{cc twin polynomial}, and the \emph{cc bracelet polynomial}. The twin Alexander polynomial descends to virtual knots.
The twin Alexander polynomial can be used to distinguish a knot in a thickened surface from its mirror image up to isotopy and can obstruct a virtual knot from being classical. Using twin Seifert matrices, we also introduce new signature invariants called \emph{twin signatures} for knots in thickened surfaces and for virtual knots. 

There are several notions of equivalence for knots in thickened surfaces, namely, identity, isotopy, diffeomorphism, and virtual equivalence. In our equivalence up to diffeomorphism, we only consider orientation-preserving diffeomorphisms of a thickened surface $F \times [-1,1]$ that map $F \times \{i\}$ to $F \times \{i\}$ for $i \in \{1, -1\}$. One can use the twin signature to distinguish a knot in a thickened surface from its mirror up to diffeomorphism and virtual equivalence, similarly to the case of classical knots.

Furthermore, we propose several natural constructions of Heegaard Floer homology for knots in thickened surfaces, and give examples why these are not invariant under stabilization. However, the Floer homology of a minimal genus representative is a virtual knot invariant. These Floer homologies can be used to detect $\varepsilon$ and $\xi$, and obstruct a knot from being a stabilization of another.

Jaeger, the second author, and Saleur~\cite{JKS}, 
Sawollek~\cite{Sawollek}, and 
Silver and  Williams~\cite{SilverWilliams} introduced 
an Alexander polynomial for virtual knots, 
which is known as the \emph{Sawollek polynomial} or the \emph{generalized Alexander polynomial}; see also Manturov~\cite{Manturov1, Manturov2}. 
This is not defined using Seifert surafaces,
but as a Yang--Baxter state sum by the work of the second author and Radford~\cite{KauffmanRadfrord}. It can also be defined via a generalization of the Burau representation of the braid group, and using Alexander biquandles, which is a generalization of the knot group; see \cite{BC, Kaip}.   
These Alexander polynomials are mirror sensitive and 
can obstruct a knot from being classical, but 
the calculation of the Sawollek polynomial is straightforward by taking a determinant associated with the diagram of the virtual knot or link.
 
Our approach, being topological, is subject to further investigation to find new ways for its calculation.
It remains to determine the relationship between the Sawollek polynomial and ours, and the relationship between the twin Alexander polynomial and Floer homology.
The fact that we define Seifert surfaces for virtual knots and associate Seifert matrices and a new signature invariant also goes beyond the earlier Alexander-type polynomials. 
 
An \emph{oriented knot}  (respectively, \emph{link})
in a connected, compact, oriented 3-manifold $M$ 
is an embedded circle 
(respectively, a disjoint union of embedded circles) in $M$.
A link in the three-sphere is called a \emph{classical} link. 
In this paper, a \emph{thickened surface} is 
$F \times [-1,1]$ for a closed, connected, and oriented surface $F$. 

Virtual knot theory is a generalization of classical knot theory, 
and studies the embeddings of circles in thickened surfaces
modulo isotopies and orientation-preserving diffeomorphisms
plus one-handle stabilizations of the surface.   
There is a natural bijection between 
the set of classical links  in the three-sphere and 
that of links embedded in the thickened two-sphere. 
Knots in thickened surfaces are said to be \emph{non-classical}  
if they cannot be made stably equivalent (see Section~\ref{secVlink}) to knots in the thickened sphere.

Knots and links in thickened surfaces have interesting properties different from classical links. 
One of them is that the linking number of two circles embedded in a thickened surface
takes values in not only integers but also half-integers; see Section~\ref{secVlink}. 

A {\it Seifert surface} $V$ 
for a link $L$ in a compact, connected, oriented 3-manifold $M$
is a compact, connected, oriented surface embedded in $M$ 
such that $L=\partial V$.  The link $L$ has a Seifert surface
if and only if $L$ is null-homologous in $M$.

Let $K$ be a  knot in a thickened surface $F\x[-1,1]$.
Then $K$ is not null-homologous 
in general. 
Even if $K$ is null-homologous and $U$ and $V$ 
are Seifert surfaces for $K$, then 
the Seifert matrices associated with $U$ and 
$V$ do not give the same Alexander polynomial in general 
(Example~\ref{prT}), unlike for classical links.
We shall overcome these issues and obtain a unique Alexander
polynomial for any knot $K$ in $F\x[-1,1]$. 
Our strategy is as follows.
 
Let $-K$ be a knot in $F\x[-1,1]$ obtained from $K$ by reversing the orientation of $K$, and  
  $K^\ast$  the mirror image of $K$ with respect to $F\x\{0\}$. 
Let  $K^{@}$ stand for $-K$ or $-K^\ast$.  (We do not consider $K^\ast$.) Consider the disjoint union $K\sqcup K^{@}$ by gluing two copies of $F \times [-1,1]$.
Then there is always a Seifert surface for $K\sqcup K^{@}$. 
We define a set of special Seifert surfaces for $K\sqcup K^{@}$ which are homologous relative to 
$K\sqcup K^{@}$; see 
Definition~\ref{defsun}, Proposition~\ref{thmmoon}, and Proposition~\ref{thmsea}.  
We call these \emph{twin Seifert surfaces} for $K \sqcup K^@$. 
Using these, we define the twin Alexander polynomials
$A^-_{K}$ and $A^{-\ast}_K$, which only depend on  
the knot type $K$ up to multiplication with $t^n$ for $n \in \Z$, and are associated with $K\sqcup -K$ 
and $K\sqcup -K^\ast$, respectively; see  
Definition~\ref{deftwin}, Theorem~\ref{thmsn}, and Proposition~\ref{pro1}. 
Note that components of Seifert matrices may be half-integers.
We also introduce the \emph{twin signatures} 
$\sigma^{-}(K)$ and $\sigma^{-\ast}(K)$ associated with these Seifert matrices.
The twin Alexander polynomial and the twin signature 
descend to virtual knots; see Section~\ref{secVAlex}.

Let $K$ be a classical knot.
If $\Delta_K$ is the classical Alexander polynomial of $K$, then 
\[
A^{-}_{K} = A^{-\ast}_{K} = \Delta_{K}^2; 
\]
see Remark~\ref{remQ}. Furthermore, if $\sigma(K)$ denotes the classical signature of $K$, then 
\[
\sigma^{-}(K) = 2\sigma(K).
\]
In fact, we have the following result for classical knots in thickened surfaces, which is an immediate corollary of Remark~\ref{remQ}.

\begin{fact}\label{facII}
Let $K$ be a knot in a thickened surface $F\x[-1,1]$ 
that lies in a 3-ball. 
Then the twin Alexander polynomial $A^@_{K}$ has the following properties, where $@ \in \{-, -\ast\}$:
\begin{enumerate}
\item All coefficients of $A^@_{K}$ are integers. 
\item Each non-zero root of $A^@_{K}$ has even  multiplicity.  
\item  $A^@_{K} = \Delta_{K\sqcup -K} = \Delta_{K\sqcup -K^\ast} = \Delta_{K}^2$.
\end{enumerate}
\end{fact}

Our main result is that the twin Alexander polynomials can obstruct a virtual knot from being classical. Furthermore, they can distinguish some knots in thickened surfaces from their mirror images up to ambient isotopy:

\begin{theorem}\label{mthhayaku}
Let $K_{m,n} \subset T^2 \times [-1,1]$ be the knot in Figure~\ref{figmn}. If $mn \neq 0$, then the following hold:
\begin{enumerate}	
\item\label{it:1} The twin Alexander polynomial $A^-_{K_{m,n}}$ has a non-zero root of multiplicity one. 
\item\label{it:2} We have
 \[
 A^-_{K_{m,n}} \neq A^{-\ast}_{K_{m,n}}. 
 \]
\end{enumerate}
Both \eqref{it:1} and \eqref{it:2} imply that $K_{m,n}$ represents a non-classical virtual knot. Furthermore, \eqref{it:2} implies that  
$K_{m,n}$ and $K^\ast_{m,n}$ are not ambient isotopic as knots in $T^2 \times [-1,1]$. 
\end{theorem}

The above theorem follows from Example~\ref{ex:non-classical} and Proposition~\ref{prop:mirror}.

\begin{remark}\label{remiroiro}
The twin Alexander polynomial distinguishes some knots in thickened surfaces from their mirror images up to ambient isotopy. Note that this is of interest since the classical Alexander polynomial does not distinguish knots from their mirror images.
Thus, our theorem is saying that the above examples have minimal genus greater than zero.

One can detect chirality of non-classical virtual knots using the twin signature. Indeed, there is a non-classical knot $K$ 
in a thickened surface such that $\sigma^{-}(K) \neq 0$, hence $K\neq K^\ast$; see Example~\ref{ex:signature}.
Of course, for the trefoil $T_{2,3}$, we have 
$\sigma^{-}(T_{2,3}) \neq 0$, so $T_{2,3} \neq T_{2,3}^\ast$.

The pair of twin Alexander polynomials $A^-$ and $A^{-*}$ can be used to distinguish a knot in a thickened surface from its mirror up to isotopy, but not up to diffeomorphism or virtual equivalence; see Example~\ref{ex:Kmm}.
\end{remark}

In Section~\ref{secmodu}, we discuss Alexander modules constructed using twin Seifert surfaces 
which are likely to be related to the twin Alexander polynomial.  
We introduce two further classes of spanning surfaces called web Seifert surfaces and bracelet Seifert surfaces for knots in thickened surfaces in Sections~\ref{sectypeII} and~\ref{sectypeIII}. 
We define an integer-valued invariant $\varepsilon(K)$ for any knot $K$ in thickened surfaces 
by maximizing the Euler characteristics of web Seifert surfaces,
and $\xi(K)$ by minimizing the genera of bracelet Seifert surfaces. We show that $\varepsilon(K) = -\xi(K)$. 
If $\varepsilon(K)=0$, then $K$ represents the virtual trivial knot. 

We introduce some variants of Floer homology associated with 
knots in thickened surfaces  in Section~\ref{secSurf} 
and virtual knots in Section~\ref{secideFHVnew}.  
The above invariant $\varepsilon(K)$ is determined by a Floer homology for $K$; see Section~\ref{secshikashi}. 

Finally, we prove that the Behrens--Golla $\delta$-invariant 
defined via Floer homology in \cite[Theorem~4.1]{BG} 
can obstruct a virtual knot from being the virtual trivial knot 
or the virtual trefoil knot; see Section~\ref{sechirameki}.

\subsection*{Acknowledgement} We would like to thank 
Sheng Bai for helpful discussions.

\section{A review of virtual knot theory}\label{secVlink}

\subsection{The definition of virtual links}\label{subsecvk}

In this subsection, we briefly review the definition of virtual knots and links \cite{Kauffman, Kauffmani}, 
and we discuss some of their properties that are relevant to the present paper. \emph{Virtual knots} 
are  the embeddings of circles in thickened 
closed, oriented surfaces
modulo isotopies and orientation-preserving diffeomorphisms
plus one-handle stabilization of the surfaces.   

By a one-handle stabilization  (respectively, destabilization), 
we mean a surgery on the surface that is performed
in the complement of the link and 
that increases (respectively, decreases) the genus of the surface. 
Note that knots and links in a thickened surface can be represented by diagrams on the surface analogously to link diagrams drawn in the plane or on the two-sphere. 
From this point of view, a one-handle stabilization is obtained by
taking two points on the surface in the link diagram complement and cutting out two disks and then adding a tube between them.
A one-handle destabilization means
cutting the surface along a curve in the complement of the link diagram and 
capping the two new boundary curves with disks. 
Handle destabilization allows the virtual knot to be eventually placed in a least genus surface in which it can be represented. 
A theorem of Kuperberg~\cite{Kuperberg} asserts that such minimal representations are topologically unique. 

\begin{figure}
\centering
\includegraphics[width=23mm]{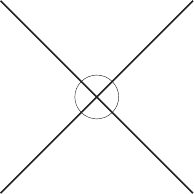}
\caption{{A virtual crossing}\label{pvcro0}}   
\end{figure}

Virtual knot theory has a diagrammatic formulation.
A \emph{virtual knot} can be represented by a \emph{virtual knot diagram} 
in $\R^2$ or $S^2$ 
containing a finite number of real crossings and some \emph{virtual crossings} 
indicated by a small circle placed around the crossing point as shown in Figure~\ref{pvcro0}.   
A virtual crossing is neither an over-crossing nor an under-crossing. A virtual crossing 
is a combinatorial structure keeping the information of the arcs of embedding
going around the handles of the thickened surface in the surface representation of the virtual link. 

\begin{figure}
\includegraphics[width=110mm]{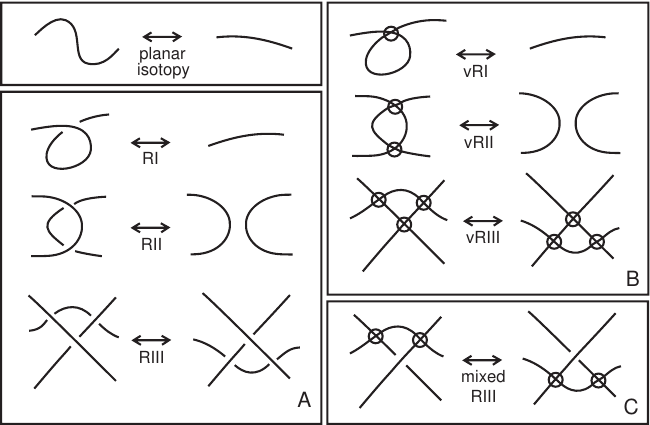}
\caption{{All Reidemeister moves for virtual links} 
\label{pall-1}}   
\end{figure}

Virtual knot and link diagrams that can be related to each other by a finite sequence of 
all Reidemeister moves   drawn in Figure~\ref{pall-1} 
are said to be
\emph{virtually equivalent} or \emph{virtually isotopic}.
The virtual isotopy class of a virtual knot diagram is called a \emph{virtual knot}.

\begin{theorem}[\cite{Kauffman, Kauffmani}]\label{pkihon} 
Two virtual link diagrams are virtually isotopic if and only if their surface embeddings are equivalent up to isotopy in the thickened surfaces, orientation-preserving diffeomorphisms of the surfaces, 
and one-handle stabilizations and destabilizations. 
\end{theorem}

A reason  why links in thickened surfaces 
and virtual links are important is as follows. 
It is an open question 
whether the Jones polynomial~\cite{Jones}, Khovanov homology~\cite{B, Khovanov1999}, and the
Khovanov--Lipshitz--Sarkar homotopy type~\cite{LSk, LSs, LSr, Seed} 
can be extended to links in three-manifolds other than the three-sphere.
This has been done in the case of links in thickened surfaces; see
\cite{APS, DKK, Kauffman, Kauffmani, KauffmanOgasasq, KauffmanNikonovOgasa, KauffmanNikonovOgasaT2, Man, MN, Igor, Ru, Tub}. 
Some of these results rely on virtual links. 
The above question is still open for other three-manifolds. 
See the discussions in \cite[Introduction and Appendix]{DKO}\cite{KauffmanOgasaquantum}.


\subsection{The linking number of virtual links}\label{subseclk}
 
In general, we cannot define the linking number for links 
in compact, oriented 3-manifolds.  
However, we can define it in the case of links in thickened surfaces 
and for virtual links. We review this construction in this section; 
see 
\cite{DKO, IMvbunrui, Kauffman, Kauffmani, Mvbunrui} for detail.  

\begin{definition}\label{defbe}
Let $J \cup K$ be a two-component link in a thickened surface $F\x [-1,1]$, 
where $J$ and $K$ may not be  null-homologous. 
We define the \emph{linking number} lk$(J,K)$ of $J$ and $K$ 
as follows.
Take the projection of $J \cup K$ into $F\x\{-1\}$.  
Assume that the projection map is a self-transverse immersion. 
We give each crossing of $J$ and $K$ a sign $+1$ or $-1$  
using the orientations of $J$, $K$ and $F\x\{-1\}$. 
The \emph{linking number} lk$(J,K)$ is 
 half of the sum of the signs
 of all crossings between $J$ and $K$.  
\end{definition}

\begin{remark} 
\begin{enumerate}	
\item The above linking number may be a half-integer. 
The linking number of the \emph{virtual Hopf link} shown in Figure~\ref{fignori} is $1/2$. In this case, $F$ is a torus and there is a single positive crossing.
If $F$ is the 2-sphere, the linking number is always an integer.

\begin{figure}
\begin{center}  
	\includegraphics{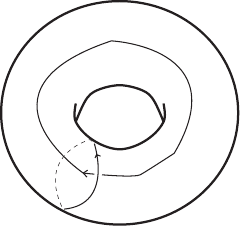}
\end{center}
\caption{{A link in the thickened torus. The linking number is 
$\frac{1}{2}$.}\label{fignori}}   
\end{figure}

\item If $F$ is the two-sphere, embed $F\x[-1,1]$ in the three-sphere naturally. 
Regard a link $J \cup K$ in the thickened two-sphere as a link in the three-sphere. 
Then the classical linking number of the link $J \cup K$ in the two-sphere 
is equal to the linking number lk$(J,K)$ in Definition~\ref{defbe}.
\end{enumerate}
\end{remark}

If a link $J \cup K$ in $F\x[-1,1]$ is virtually equivalent to 
a link $J' \cup K'$ in $F'\x[-1,1]$, 
then we have lk$(J,K) = \text{lk}(J',K')$.   
Therefore, the linking number descends to virtual links:  

\begin{definition}\label{defVlk0}
Let $\mathcal J \cup  \mathcal K$ be 
a two-component virtual link represented by 
a link $J \cup K$ in $F\x[-1,1]$. 
Define the \emph{linking number} lk$(\mathcal J, \mathcal K)$ 
to be lk$(J,K)$.
\end{definition} 

Let $P \cup Q$ be a virtual link diagram in the plane that represents a virtual link 
$\mathcal P \cup \mathcal Q$. 
Each classical crossing of $P$ and $Q$ is oriented using 
the orientations of $P$, $Q$, and the plane. 
Assign to a positive classical crossing the sign $+1$, 
to a negative classical crossing $-1$, and to a virtual crossing $0$.  Then the linking number  lk$(\mathcal P, \mathcal Q)$  is
half of the sum of the signs associated with each crossing.  

Consider the diagrams of virtual Hopf links shown in Figure \ref{figVHopf}. 
The linking number of the virtual link on the left is 
$-1/2$. The linking number of the virtual link on the right is $1/2$. 

\begin{figure}
\centering
\includegraphics[width=120mm]{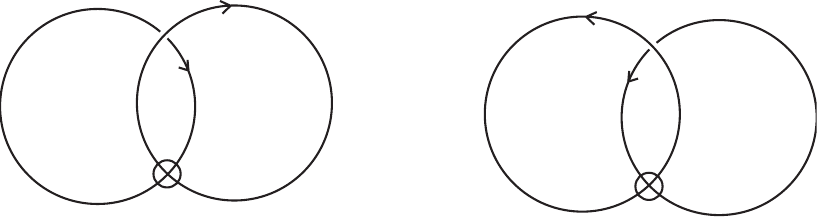}
\caption{{Virtual Hopf links}\label{figVHopf}}   
\end{figure}

\section{Twin Seifert surfaces  for links in thickened surfaces}\label{secwakatta}

\subsection{The $d$-th Alexander polynomial associated with a Seifert surface}\label{subsecdA}

We will introduce new Alexander polynomials 
for all links in thickened surfaces. 
However, we begin with the case of 
null-homologous links; compare with~\cite{BCG, BGHNW}.

\begin{definition}
	Let $L$ be a null-homologous link in a thickened surface $\mathcal F=F\x [-1,1]$. Then a \emph{Seifert surface} $V$ for $L$ is a compact, connected, and oriented surface embedded in the interior of $\mathcal{F}$ with boundary $L$.
	
	Let $\{b_1, \dots, b_n\}$ be a basis of $H_1(V)$, and choose oriented simple closed curves $\beta_1, \dots, \beta_n$ in $V$ representing $b_1, \dots, b_n$, respectively . Then the \emph{Seifert matrix} of $V$ with respect to this basis is a $n \times n$ matrix with $(i,j)$-th entry $\text{lk}(\beta_i, \beta_j^+)$, where $\beta_j^+$ is the positive push-off of $\beta_j$, and the linking number is the one given in Definition~\ref{defbe}. Note that components of a Seiferet matrix may be half-integers.
\end{definition}

Suppose that the link  $L$ in a thickened surface $\mathcal F=F\x [-1,1]$ has a Seifert surface $V$. 
Then, as in the case of classical links (see Crowell--Fox~\cite[p.~119]{CF}), we can consider a {\it Seifert matrix} $X$, and use it to define the {\it $d$-th Alexander polynomial} $\Delta_{K, V, d}$ associated with $V$ and $d \in \Z$. Let $n$ be the size of $X$. If $0 \le d < n$, then it is the greatest common divisor of the $(n - d) \times (n - d)$ minors of $tX - X^T$ in $\Z[1/2][t]$. (Alternatively, it is a generator of the smallest principal ideal containing the ideal generated by all $(n-d) \times (n-d)$ minors.)
If $d < 0$, then we let $\Delta_{K, V, d} := 0$. For $d > n$, we write $\Delta_{K, V, d} := 1$.
This only depends on the {\it S-equivalence} class of the Seifert matrix; see \cite{Levinesimp, LevineAct, Murasugi, Trotter}.  
(Levine~\cite{Levinesimp} uses the term ``equivalence'' for  ``S-equivalence.'') 
The zeroth Alexander polynomial is known simply as the {\it Alexander polynomial}. Note that some authors use $(n-d+1) \times (n-d+1)$ minors instead of $(n-d) \times (n-d)$ minors to define $\Delta_{K, V, d}$; compare \cite[Remark~11.5.1]{OSS} with \cite[Definition~6.2]{Lickorish}. The ring $\Z[1/2]$ is a Euclidean domain, since it is the localization of $\Z$ along the multiplicative set consisting of powers of~$2$.
 
\begin{definition}
Two polynomials $f(t)$, $g(t)\in\Z[1/2][t]$  
are said to be 
\emph{balanced equivalent} (written $f \dot{=} g$)
if there is 
are integers $n$ and $k$ such that $\pm 2^k t^nf(t)=g(t)$. 
We say that $f(t)$ and $g(t)$ represent the same {\it balance class}.
\end{definition}

\begin{proposition}\label{prV}
Let $L$ and $V$ be as above.  
Assume that another Seifert surface $V'$ is obtained from $V$ by a surgery along an embedded 1-handle in $\mathcal F$. 
Then any Seifert matrix associated with $V$ is $S$-equivalent to that associated with $V'$. 
Therefore, the $d$-th Alexander  polynomial defined by each represents the same balance class. Furthermore, the zeroth Alexander polynomial is well-defined up to multiplication by $t^n$ for $n \in \Z$. 
\end{proposition}

The proof is a natural generalization of the analogous result for classical links; see \cite{Levinesimp, LevineAct, Murasugi, Trotter}.
 
\begin{example}\label{prT}   
Let $L$ be a link in a thickened surface
and let $U$ and $V$ be Seifert surfaces for $L$.  
In general, the Alexander polynomials defined by $U$ and $V$ are not balanced equivalent.
Indeed, let $U$ be a disc in $T^2\x\{0\}\subset T^2\x[-1,1]$, 
and let $L :=\partial U$. Then $L$ is the trivial knot in $T^2\x[-1,1]$.    
Let 
\[
V := (T^2\x\{0\}) \setminus \text{Int}(U);
\] 
see Figure~\ref{figT2}. 
Note that $U$ and $V$ are Seifert surfaces for $L$.

\begin{figure}
	\centering\includegraphics[width=40mm]{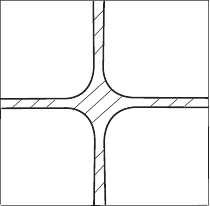}
	\caption{The trivial knot $K$ in the thickened torus, and its Seifert surface $V$.}\label{figT2}   
\end{figure}
    
A Seifert matrix defined by $U$ is the empty matrix. 
A Seifert matrix $X$ defined by $V$ is 
\[
\left(
    \begin{array}{cc}
0& 1/2\\
-1/2 &0\\
    \end{array}
  \right). 
\]
The Alexander polynomials associated with $U$ and $V$ are the balance classes of 1 and 
$\det (t X-X^T) =\frac{1}{4}(t+1)^2$, respectively,
where $X^T$ is the transpose of $X$. 
\end{example}

Let $L$ be a link in a thickened surface and
$U$ and $V$  Seifert surfaces for $L$.     
We may have $(U \setminus L)\cap(V \setminus L) \neq \emptyset$.  
Suppose that the two-cycle $U \cup -V$ is null-homologous in  $\mathcal F$.  
Then $U$ is obtained from $V$ by a sequence of surgeries along embedded one-handles and two-handles. 
This is proved using the same method as
\cite[Theorem~1]{Levinesimp}. 
By Proposition~\ref{prV}, we have the following.

\begin{proposition}\label{prNl}
Let $L$, $U$, and $V$ be as above. 
Then any Seifert matrices associated with $U$ and $V$ are $S$-equivalent. 
Therefore, the $d$-th Alexander  polynomial defined by each represents the same balance class.
\end{proposition}

\begin{definition}\label{def:type}
Let $L$ be a null-homologous link in the thickened surface $F \times [-1, 1]$ (we do not consider $L$ up to isotopy now). 
Then there is a two-chain $\varsigma \in C_2(F \times [-1,1])$ such that $\partial\varsigma=L$,   
 where $\partial$ denotes the differential acting on chains.  
Let $p \in F$ be a point such that $(\{p\} \times [-1,1]) \cap L = \emptyset$. 
Suppose that $\{p\} \x  [-1,1]$ and $\varsigma$ intersect transversely
algebraically $n$ times. 
Then we call $\varsigma$ a {\it type $n$ two-chain for $L$ with respect to $p$}. 
\end{definition}

\begin{claim}\label{cla2ch} 
Let $L$ and $p$ be as above. 
Assume that $U$ and $V$
are type $n$ Seifert surfaces for $L$ with respect to $p$.
\begin{enumerate}
\item \label{it:intersection}
Then the two-cycle $\gamma = U \cup -V$  
and $\{p\} \x  [-1,1]$ 
have zero algebraic intersection number. 
That is, the following intersection product is zero: 
\[
\begin{split}
H_1(F\x [-1,1], F\x\{-1,1\})\x H_2(F) &\to \Z \\
[\{p\} \x [-1,1], \{p\} \x\{-1,1\}] \cdot\gamma &\mapsto 0.
\end{split}
\]
Furthemore, $\gamma$ is null-homologous in $F\x [-1,1]$. 

\item\label{it:S-equivalence} 
Any Seifert matrices associated with $U$ and $V$ are $S$-equivalent. 
Therefore the $d$-th Alexander  polynomial defined by each represents the same balance class.
\end{enumerate}
\end{claim} 

\begin{proof} 
The arc $a := \{p\} \x  [-1,1]$ has algebraic intersection number $n$ with both $U$ and $V$. Hence $a$ and $\gamma$ 
have zero algebraic intersection number. Since 
\[
H_1(F\x [-1,1], F\x\{-1,1\}) \cong \Z 
\]
and is generated by $[a, \partial a]$,
it follows from Poincar\'e--Lefschetz duality that $\gamma$
is null-homologous, which concludes the proof of part~\eqref{it:intersection}. 
Part~\eqref{it:S-equivalence} now follows from Proposition~\ref{prNl}
\end{proof}

\subsection{Twin Seifert surfaces}\label{subsec}

Let $K$ be an oriented  knot in a thickened surface $\mathcal F=F\x[-1,1]$. We do not require $K$ to be null-homologous 
in $\mathcal F$, and do not consider it up to isotopy for now.

Consider $-K$, which is $K$ with its orientation reversed. 
We denote by $\mathcal F_{-K}$ a copy of $\mathcal F$ that includes $-K$.
In $\mathcal F \sqcup \mathcal F_{-K}$, identify $F\x\{-1\} \subset \mathcal F$ with $F\x\{1\} \subset \mathcal F_{-K}$, 
and call the result $F\x I$, where $I$ is obtained by gluing two copies of $[-1,1]$. We obtain the link $K\sqcup -K$ in $F\x I$. 

If we reverse the orientation of $[-1,1]$ in  $\mathcal F$,
then the orientation of  $\mathcal F$ is reversed. 
Call this manifold $\mathcal F_{-K^\ast}$.
If we reverse the orientation of $K$ as well, 
then we obtain the knot $-K^\ast$ in $\mathcal F_{-K^\ast}$.
In $\mathcal F \sqcup \mathcal F_{-K^\ast}$, identify
$F\x\{-1\} \subset \mathcal F$ 
with $F\x\{-1\} \subset \mathcal F_{-K^\ast}$. 
Call the result $F\x I$, where $I$ is made from two copies of $[-1,1]$. 
We obtain the link $K\sqcup -K^\ast$ in $F\x I$.
(Note the difference between 
$\mathcal F \cup \mathcal F_{-K}$ 
 and  $\mathcal F \cup\mathcal F_{-K^\ast}$.)  

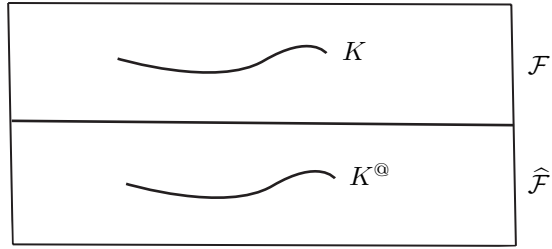
\begin{figure}
\begingroup%
  \makeatletter%
  \providecommand\color[2][]{%
    \errmessage{(Inkscape) Color is used for the text in Inkscape, but the package 'color.sty' is not loaded}%
    \renewcommand\color[2][]{}%
  }%
  \providecommand\transparent[1]{%
    \errmessage{(Inkscape) Transparency is used (non-zero) for the text in Inkscape, but the package 'transparent.sty' is not loaded}%
    \renewcommand\transparent[1]{}%
  }%
  \providecommand\rotatebox[2]{#2}%
  \newcommand*\fsize{\dimexpr\f@size pt\relax}%
  \newcommand*\lineheight[1]{\fontsize{\fsize}{#1\fsize}\selectfont}%
  \ifx\svgwidth\undefined%
    \setlength{\unitlength}{237.48943801bp}%
    \ifx\svgscale\undefined%
      \relax%
    \else%
      \setlength{\unitlength}{\unitlength * \real{\svgscale}}%
    \fi%
  \else%
    \setlength{\unitlength}{\svgwidth}%
  \fi%
  \global\let\svgwidth\undefined%
  \global\let\svgscale\undefined%
  \makeatother%
  \begin{picture}(1,0.38887011)%
    \lineheight{1}%
    \setlength\tabcolsep{0pt}%
    \put(0,0){\includegraphics[width=\unitlength,page=1]{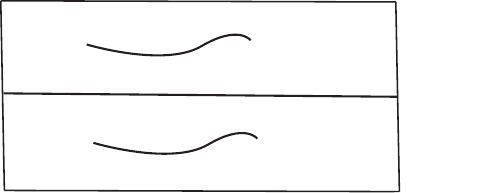}}%
    \put(0.53152924,0.29704542){\makebox(0,0)[lt]{\lineheight{1.25}\smash{\begin{tabular}[t]{l}$K$\end{tabular}}}}%
    \put(0.54279498,0.09801712){\makebox(0,0)[lt]{\lineheight{1.25}\smash{\begin{tabular}[t]{l}$K^@$\end{tabular}}}}%
    \put(0.82631642,0.27451389){\makebox(0,0)[lt]{\lineheight{1.25}\smash{\begin{tabular}[t]{l}$\mathcal F$\end{tabular}}}}%
    \put(0.82631642,0.08487383){\makebox(0,0)[lt]{\lineheight{1.25}\smash{\begin{tabular}[t]{l}$\widehat{\mathcal{F}}$\end{tabular}}}}%
  \end{picture}%
\endgroup%

\caption{{The link 
 $K\sqcup K^{@}$ in $F\x I=\mathcal F\cup\widehat{\mathcal F}$. 
}\label{figshogatsu}
}   
\end{figure}

\begin{figure}
	\includegraphics[width=3cm]{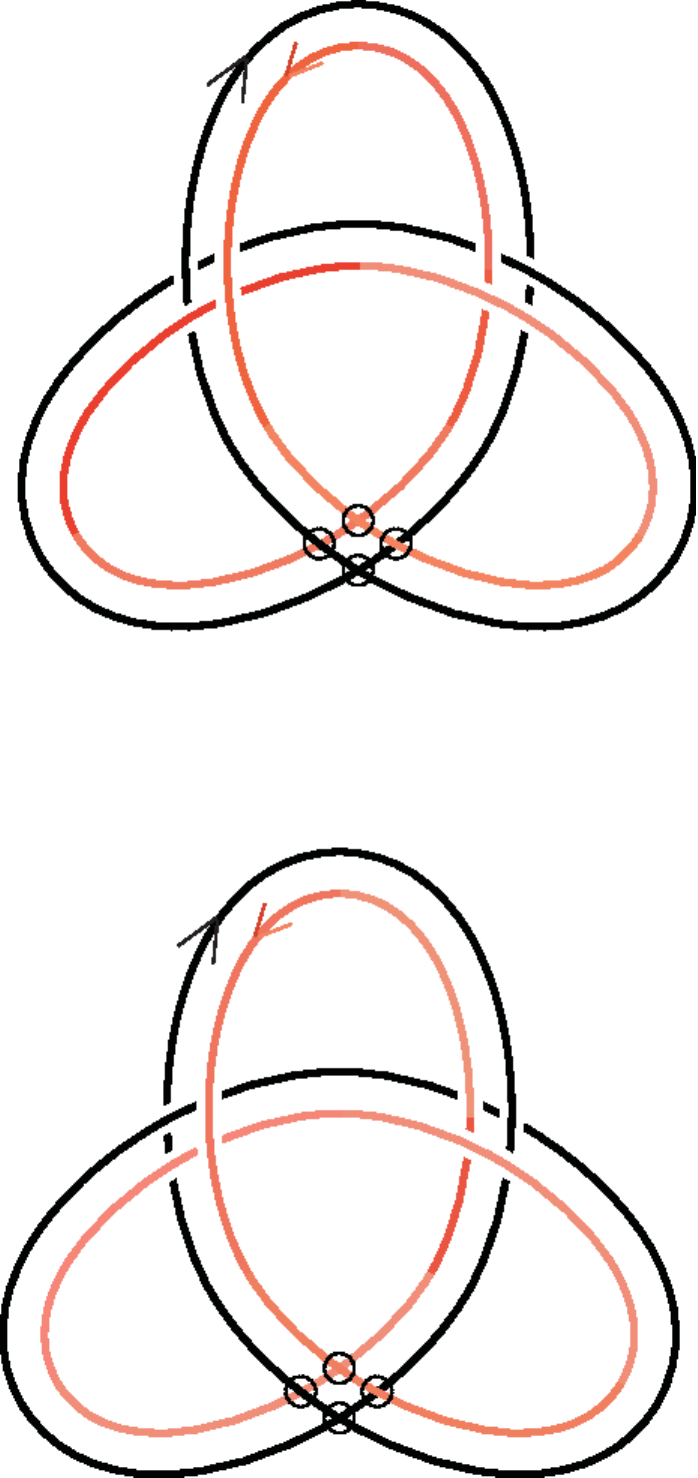}
	\caption{Virtual diagrams of the links $K \sqcup -K$ (top) and $K \sqcup -K^*$ (bottom) in $T^2 \times I$, where $K$ is the component in red.}\label{fig:trefoil}
\end{figure}

Let $K^{@}$ stand for either $-K$ or $-K^\ast$. 
Let 
 $\widehat{\mathcal F}$ 
denote $\mathcal F_{K^{@}}$. 
Then $F\x I=\mathcal F\cup \widehat{\mathcal F}$; 
see Figure~\ref{figshogatsu} for a schematic picture. 
Since  $K\sqcup K^{@}$ is a vanishing one-cycle in  $F\x I$, 
there is a Seifert surface $W$ for $K\sqcup K^{@}$.
See Figure~\ref{fig:trefoil} for a concrete example that shows $K \sqcup -K$ and $K \sqcup -K^*$ in the case when $F = T^2$ and $K$ is the virtual trefoil.

\begin{remark} 
We do not consider the case where $K^{@}=K^\ast$, since $K\sqcup K^\ast$ is not null-homologous in $F\x I$. 
Therefore $K\sqcup K^\ast$ does not admit a Seifert surface.
We need a Seifert surface in order to define the twin Alexander polynomial.
\end{remark}

The following result is used in the proof of Theorem~\ref{mthhayaku}.

\begin{proposition}\label{prop:mirror}
	Let $K$ be a knot in the thickened surface $F \times [-1,1]$, and let $K \sqcup K^*$ and $K \sqcup -K^*$ be as above. Then the following are equivalent:
	\begin{enumerate}
		\item\label{it:one} the knots $K$ and $K^*$ are isotopic,
		\item\label{it:two} the knots $-K$ and $-K^*$ are isotopic,
		\item\label{it:three} the links $K \sqcup -K$ and $K \sqcup -K^*$ are isotopic, where this isotopy does not necessarily preserve the order of the components. 
	\end{enumerate}
\end{proposition}

\begin{proof}
	It is clear that \eqref{it:one} implies \eqref{it:two} and \eqref{it:two} implies \eqref{it:three}. To see that \eqref{it:three} implies \eqref{it:one}, first suppose that the isotopy preserves the order of the components of $K \sqcup -K$ and $K \sqcup -K^*$. Then $-K$ and $-K^*$ are isotopic, so $K$ and $K^*$ are also isotopic. If the isotopy reverses the order of the components, then $K$ and $-K^*$ are isotopic and $-K$ and $K$ are isotopic. By the latter, $-K^*$ is isotopic to $K^*$, which implies that $K$ is isotopic to $K^*$. 
\end{proof}

\begin{lemma}\label{lem:well-defined}
	Let $W \subset F \times I$ be a Seifert surface for the link $L := K \sqcup K^@$, and choose a point $p \in F$ such that 
	$(\{p\} \times I) \cap L = \emptyset$.
	Then the type of $W$ in the sense of Definition~\ref{def:type} is independent of the choice of $p$.  
\end{lemma}

\begin{proof}
	Let $\pi \colon F \times I \to F$ be the projection.
    The type of $W$ with respect to $p$ is the algebraic intersection number of $W$ and $\{p\} \times I$. This is the same as the multiplicity of the two-chain $\pi(W)$ at $p$. However, $\pi(W)$ is a two-cycle since $\pi(K) = -\pi(K^@)$, and is hence homologous to $n[F]$ for some $n \in \Z$. This $n$ is the type of $W$.
\end{proof}

\begin{definition}\label{defsun}
Let $K$, $K^{@}$, $\mathcal F$, and $\widehat{\mathcal F}$
be as above, where $K^{@} \in \{-K, -K^\ast\}$. 
If $W$ is a type zero Seifert surface for $K \sqcup K^{@}$, 
then we call $W$ a {\it twin Seifert surface} for $K\sqcup K^{@}$. 
\end{definition}

Note that we do not consider $K$ or $K^{@}$ up to isotopy for now. 
Seifert surfaces are connected by definition. In particular, 
twin Seifert surfaces are connected. 
The existence of twin Seifert surfaces follows from the following result:

\begin{proposition}\label{thmmoon}
Let $K$ be a knot in $F\x[-1,1]$, and construct the link $K \sqcup K^@$  in $F\x I$ as above for $K^@ \in \{-K, -K^*\}$. 
For any integer $n$, there is a Seifert surface for $K\sqcup K^{@}$ of type $n$. 
 \end{proposition}  

\begin{proof}
	Write $L := K \sqcup K^@$, and let $\pi \colon F \times I \to F$ be the projection. 
	Since $L$ is null-homologous, there is a two-chain $\zeta \in C_2(F \times I)$ such that $\partial \zeta = L$. Then $\pi(\zeta)$ is homologous to $k[F]$ for some integer $k$. Let $W$ be a Seifert surface for $L$ representing the homology class 
	\[
	[\zeta + (n-k)F] \in H_2(F \times I, L).
	\]
	Then $W$ is type $n$.
\end{proof}

\begin{remark}
We will describe two explicit methods to construct a twin Seifert surface in Sections~\ref{subsec2ways} and~\ref{subsec2waysII}.
We will introduce two other types of Seifert surfaces:
web Seifert surfaces in Section~\ref{sectypeII} 
and bracelet  
Seifert surfaces in Section~\ref{sectypeIII}.
\end{remark}

\subsection{Constructing twin Seifert surfaces, method 1}\label{subsec2ways}

\begin{figure}
	\includegraphics[width=40mm]{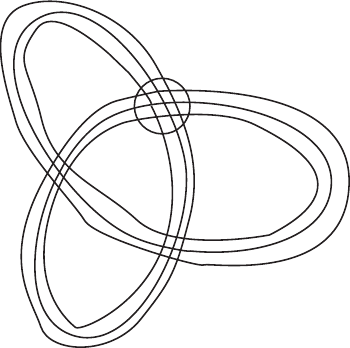}
	\caption{An example of the projection $P$ of $K$ to $F$ drawn using virtual crossings, 
		and the immersed tubular neighborhood $N$ of $P$.}\label{figN}
\end{figure}

Let $K$ be an oriented  knot in a thickened surface 
$\mathcal F=F\x[-1,1]$. 
Let $P$ be the projection of $K$ to $F$.
If we take the oriented smoothing of the crossings of $P$, we obtain an oriented one-manifold $S \subset F$.
Let $N$ be an immersed tubular neighborhood of $P$ in $F$, as in Figure~\ref{figN}, where we draw $P$ and $N$ in the plane using virtual crossings. 
There are compact oriented surfaces $R_\pm$ in $N\x[-1,1]$ 
such that 
\[
\partial R_\pm = K \sqcup (-S \times \{\pm 1\}).
\]
Indeed, attach a half-twisted band to $K$ over each crossing of $P$
that lies in $N \times [-1,1]$, as shown in the left of Figure~\ref{figVAlex}. We write $B$ for the union of the bands. 
Then
\[
R_- := (S \x [-1,0]) \cup B \text{ and } R_+ := (S \times [0,1]) \cup B.
\] 

\begin{figure}
	\includegraphics{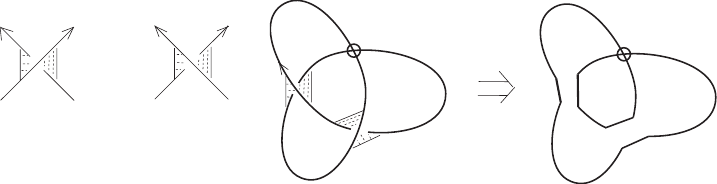}
	\caption{Attaching bands near classical crossings}\label{figVAlex}  
\end{figure}

Consider $K^{@} \in \{-K, -K^\ast\}$ in $\widehat{\mathcal F}=F\x[-1,1]$.
The projection of $K^@$ to $F$ is $-P$, so we obtain $-S$ if we smooth the crossings.
We construct the surface $R^@_\pm \subset \widehat{\mathcal F}$ analogously to $R_\pm$.

Recall that we identify $F\x\{-1\} \subset \mathcal F$ with $F \x \{1\} \subset \mathcal F_{-K}$ when $K^@ = -K$, and with 
$F\x\{-1\} \subset \mathcal F_{-K^\ast}$ when $K^@ = -K^\ast$. 
The result is $F \x I$, where $I$ is the union of two copies of $[-1,1]$.

When $K^@ = -K$, the union $R_- \cup R^@_+$ is a twin Seifert surface  for $K\sqcup -K$ in $F\x I$. 
When $K^@ = -K^\ast$, the union $R_- \cup R^@_-$ is a twin Seifert surface  for $K\sqcup -K^\ast$ in $F\x I$.  
Note that in both cases the Seifert surface lies in $N \x I$, and is hence type zero.

\subsection{Constructing twin Seifert surfaces, method~2}\label{subsec2waysII}
 
\begin{figure}
\includegraphics[width=136mm]{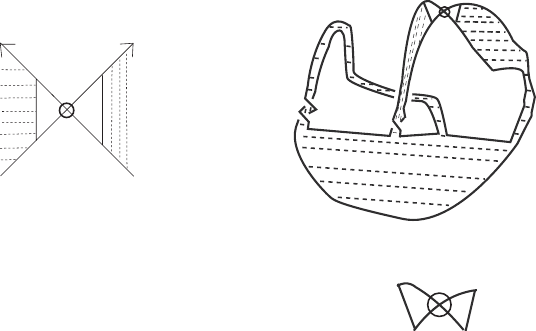}
\caption{{An illustration of method~2 for constructing twin Seifert surfaces. The surface $A$ is shown in the upper right.
}\label{figVAlexY}}   
\end{figure}

Let $K$ be a virtual knot.
We start with a diagram of $K$ with virtual crossings in $\R^2$.
We resolve each virtual crossing as in the left of Figure~\ref{figVAlexY}, by attaching two line segments. Then we get a classical link diagram $D$ (i.e., there are no virtual crossings); see the upper right of Figure~\ref{figVAlexY}.  
Let $A$ be a Seifert surface for $D$. Let $F$ be a surface obtained from $S^2 = \R^2 \cup \{\infty\}$ by attaching a one-handle at each virtual crossing. Then $K$ and $A$ naturally embed into $\mathcal F := F \times [-1,1]$; see the top of Figure~\ref{fig:12}. 

We obtain $A^@$ in $\widehat{\mathcal F}=F\x[-1,1]$ analogously. 
Connect $A$ and $A^@$ by tubes over
the unknots around the virtual crossings of $K$ shown in the lower right of Figure~\ref{figVAlexY}. If there are no virtual crossings,
then we connect $A$ and $A^@$ by 
a surgery along a vertical 3-dimensional 1-handle. 
The result is a twin Seifert surface $W$ for $K\sqcup K^{@}$ in $F\x I$. It is type zero because $W$ is disjoint from $\{\infty\} \times I$.
See Example~\ref{ex:signature} for a concrete application of this method.

\section{The twin Alexander polynomial for knots in thickened surfaces}\label{secAlex}

\begin{proposition}\label{thmsea}  
Let $K$ be a knot in a thickened surface $\mathcal F=F\x[-1,1]$, 
and let $K^@ \in \{-K, -K^*\}$.
Then the S-equivalence class of Seifert matrices  
associated with twin Seifert surfaces for $K\sqcup K^{@}$ is an invariant of the diffeomorphism class of $K$ in $\mathcal F$.
\end{proposition}


\begin{proof}
Construct $\widehat{\mathcal F}$ and $F\x I=\mathcal F\cup \widehat{\mathcal F}$ as before. Let $W$ be a twin Seifert surface for $K \sqcup K^@$.
Let $\iota \colon \mathcal F \to \mathcal F$ be a diffeomorphism and write $K_1 := \iota(K)$.
Consider $K_1^@ \in \{-K_1, -K_1^\ast\}$ in $\widehat{\mathcal F}$.
Take a twin Seifert surface $W_1$ for $K_1 \sqcup K_1^{@}$ in $F\x I$. It suffices to show that the S-equivalence classes of Seifert matrices associated with $W$ and $W_1$ agree.

By the symmetry between $\mathcal F$ and $\widehat{\mathcal F}$,  there is an automorphism $\iota^@$ of $\widehat{\mathcal F}$ taking $K^@$ to  $K_1^@$.
We can glue $\iota$ and $\iota^{@}$ to obtain an automorphism $\jmath$ of $F \times I$ such that $\jmath|_{\mathcal F} = \iota$ and $\jmath|_{\widehat{\mathcal F}} = \iota^@$.
By construction, $\jmath(W)$ is a twin Seifert surface for $K_1 \sqcup K_1^{@}$. The result now follows from part~\eqref{it:S-equivalence} of Claim~\ref{cla2ch}. 
\end{proof}

\begin{definition}\label{deftwin}
Let $K$ be a knot in a thickened surface and let $K^@ \in \{-K, -K^\ast\}$. 
Let $X$ be a Seifert matrix obtained from a twin Seifert surface 
for $K\sqcup K^{@}$. We call $X$ a \emph{twin Seifert matrix}.  
The $d$-th Alexander polynomial for $d\in\Z$ associated with $X$ 
is called the {\it $d$-th twin Alexander polynomial} 
$A_{K, d}^@$ of $K$. In particular, the \emph{twin Alexander polynomial} is $A^@_K := A^@_{K, 1}$.
The {\it twin signature} $\sigma^@$ of $K$ is the signature of $X + X^T$.  
There are two kinds of Alexander polynomials and signatures   
since $K^@$ is either $-K$ or $-K^\ast$. 
\end{definition}

\begin{remark}
	Note that the twin Alexander polynomial is defined as the \emph{first} Alexander polynomial of X, not the zeroth. So it is the greatest common divisor of the $(n-1) \times (n-1)$ minors of $tX - X^T$ in $\Z[1/2][t]$ if $X$ is an $n \times n$ matrix; see Crowell--Fox~\cite{CF}.
	The zeroth Alexander polynomial is also defined, and is an isotopy invariant.
\end{remark}

By Proposition~\ref{thmsea}, 
we have the following. 

\begin{theorem}\label{thmsn}
 The balance class of $d$-th twin Alexander polynomial and the twin signature 
 are topological invariants of knots in thickened surfaces. 
\end{theorem}

\begin{remark}\label{remyy}
The twin signature $\sigma^{-}(K)$ is a nontrivial invariant of $K$.  
For example, if $K$ is a classical knot in $S^2 \times [-1,1]$, then  $\sigma^{-}(K) = 2\sigma(K)$.  
See Example~\ref{ex:signature} for a non-classical example. 
\end{remark}

\begin{proposition}\label{pro1} 
Let $K$ be a knot in a thickened surface. Then there is a polynomial $f(t) \in \Z[1/2][t]$ that represents the twin Alexander polynomial of $K$ such that $f(1)= \pm 2^k$ for some $k \in \Z$. 
\end{proposition}

\begin{proof}
Let $W$ be a twin Seifert surafce for $K\sqcup K^{@}$. 
Choose a basis $\mathcal{B} = \{b_1, \dots, b_n\}$ of $H_1(W)$ such that $b_1, \dots, b_{n-1}$ is symplectic with respect to the intersection form of $W$ and $b_n = [K]$. Let $X$ be the corresponding Seifert matrix. We claim that  $X-X^T$ is the matrix $M$ of the intersection form on $H_1(W)$ in the basis $\mathcal{B}$. Indeed, let $\beta_i$ be an oriented simple closed curves on $W$ whose homology class is $b_i$ for $i \in \{1, \dots, n\}$. For $i$, $j \in \{1, \dots, n\}$, the $(i,j)$-th entry of $X - X^T$ is 
\[
\lk(\beta_i, \beta_j^+) - \lk(\beta_j, \beta_i^+) = \lk(\beta_i, \beta_j^+ \sqcup -\beta_j^-), 
\]
where $\beta_j^-$ is the negative push-off of $\beta_j$. 

Let $B$ be an annulus in $F \times I$ with boundary $\beta_j^+ \sqcup -\beta_j^-$ that intersects $W$ transversely in $\beta_j$. We choose $B$ to be very thin. Let $\pi \colon F \times I \to I$ be the projection. By perturbing $\beta_i$ and $B$, we can arrange that every component of $\pi(\beta_i) \cap \pi(B)$ is an arc $a$ connecting $\pi(\beta_j^+)$ and $\pi(\beta_j^-)$. If $a$ contains the projection of a point $p$ of $\beta_i \cap \beta_j$, then a local computation shows that the intersection sign at $p$ agrees with the average of the crossings signs at $\partial a$ (the crossings of $\pi(\beta_i)$ with $\pi(\beta_j^+)$ and $-\pi(\beta_j^-)$ along $a$ consist of  an undercrossing and an overcrossing in some order, but $\pi(\beta_j^+)$ and $-\pi(\beta_j^-)$ point in opposite directions, so the crossing signs agree). If $a$ does not contain the projection of a point of $\beta_i \cap \beta_j$, then both crossings at $\partial a$ are either undercrossings or they are both overcrossings, and as $\pi(\beta_j^+)$ and $-\pi(\beta_j^-)$ point in opposite directions, the crossing signs cancel. We conclude that 
\[
\lk(\beta_i, \beta_j^+ \sqcup \beta_j^-) = \beta_i \cdot \beta_j,
\]
as claimed.

The twin Alexander polynomial $f$ of $K$ is the greatest common divisor of the $(n-1) \times (n-1)$ minors of $tX - X^T$. As $X - X^T = M$, every such minor evaluates at $1$ to the corresponding minor of $M$. As $b_i \cdot b_n = 0$ for any $i \in \{1, \dots, n-1\}$, each minor of $M$ containing $b_n$ is zero. Furthermore, the minor corresponding to $\{b_1, \dots, b_{n-1}\}$ is $1$. Hence $f(1)$ is a unit in $\Z[1/2]$, so $\pm 2^k$ for some $k \in \Z$.
\end{proof}

By Proposition~\ref{pro1}, we can normalize the twin Alexander polynomial such that $A_K^@(1) = 1$, which gets rid of the $\pm 2^k$ ambiguity of balanced equivalence for $k \in \Z$. This still leaves an ambiguity of multiplication by $t^n$ for $n \in \Z$.

\begin{remark}\label{remQ} 
Let $\Delta_{L, d}$ be the $d$-th Alexander polynomial of a link $L$ in $S^3$. If $K$ is a knot in $S^3$, let $K^@ \in \{-K, -K^\ast\}$. 
Then $\Delta_{K \sqcup K^{@}, 1} = \Delta_{K, 0}^2$. 
Indeed, let $V$ be a Seifert surface for $K$, and let $V^@$ be the corresponding Seifert surface for $K^@$. Let $W$ be the connected sum of the components of $V \sqcup V^@$, which is a Seifert surface for $K \sqcup K^@$. Choose a basis $\mathcal{B}$ of $H_1(V)$, and let $\mathcal{B}^@$ be the corresponding basis of $H_1(V^@)$. We write $X$ and $X^@$ for the Seifert matrices of $V$ and $V^@$ using the bases $\mathcal{B}$ and $\mathcal{B}^@$, respectively. Let $W$ be the connected sum of $V$ and $V^@$, and write $\mu$ for the meridian of the connected sum tube. Then $\mathcal{B}_W := \mathcal{B} \cup \mathcal{B}^@ \cup \{[\mu]\}$ is a basis of $H_1(W)$. As both $\mu$ and $\mu^+$ are unlinked from any element of $\mathcal{B} \cup \mathcal{B}^@$, the row and column of $[\mu]$ vanish in the Seifert matrix $X_W$ of $W$ with respect to the basis $\mathcal{B}_W$. Hence, $X_W = X \oplus X^@ \oplus 0$. It follows that the only non-zero minor of $tX_W - X_W^T$ is 
\[
\det\left(t(X \oplus X^@) - (X \oplus X^@)^T \right) = \Delta_{K,0} \Delta_{K^@,0}. 
\]
As the classical Alexander polynomial is invariant under mirroring and orientation reversal, we obtain that $\Delta_{K \sqcup K^{@}, 1} = \Delta_{K, 0}^2$. Since we have $\Delta_{K,0}(1) = 1$, the polynomial
$\Delta_{K \sqcup K^{@}, 1}$ recovers $\Delta_{K, 0}$. 

Furthermore, if $K$ is a classical knot, then $\Delta_{K \sqcup K^{@}, 0} = 0$ since $K \sqcup K^@$
is a split link. So $K$ is non-classical whenever $A^@_{K, 0} \neq 0$.
\end{remark}

Fact~\ref{facII} from the introduction is now an immediate corollary of Remark~\ref{remQ}.

\section{The twin Alexander polynomial for virtual knots}\label{secVAlex} 

In this section, we prove that the $d$-th twin Alexander polynomial and the twin signature are also virtual knot invariants. 

\begin{theorem}\label{thmVAle} 
 The balance class of the $d$-th twin Alexander polynomial for any $d \in \N$ and the twin signature are invariant under virtual one-handle
 stabilization and destabilization.
\end{theorem}

\begin{proof}
Let $K$ be a knot in a thickened surface $F \times I$.
By Section~\ref{subsec2ways}, there is a Seifert surface $W$ for $K \sqcup K^@$ 
that lies inside $N \times I$  for $N$ an immersed tubular neighbourhood 
of the projection of $K$ to $F$. 
Hence, we can do a stabilization or destabilization of $F$ in the complement of $W$.	 
So $K$ has the same twin Seifert surface with the same Seifert matrix 
before and after the stabilization or destabilization, 
and the $d$-th twin Alexander polynomial and the twin signature remain unchanged.
\end{proof}

Hence, the following is well-defined:

\begin{definition}\label{defVAle} 
Let $\mathcal K$ be a virtual knot and $@ \in \{-, -\ast\}$. 
Let $K$ be a knot in a thickened surface that represents $\mathcal K$. 
The \emph{$d$-th twin Alexander polynomial of $\mathcal K$} is
$A_{\mathcal K, d}^@ := A_{K, d}^@$.
The \emph{twin signature of $\mathcal K$} is
$\sigma^@(\mathcal K) := \sigma^@(K)$.    
\end{definition}

We have the following by Fact~\ref{facII} and  Remark~\ref{remQ}. 

\begin{fact}\label{facI}
	Suppose that  the virtual knot $\mathcal K$ can be represented by a classical knot and $@ \in \{-, -\ast\}$. 
	Then the twin Alexander polynomial $A^@_{\mathcal K}$ has the following properties:
	\begin{enumerate}
		\item All coefficients of $A^@_{\mathcal K}$ are integers. 
		\item\label{it:roots} Each non-zero root of $A^@_{\mathcal K}$ has even multiplicity.  
		\item\label{it:chiral}  $A^-_{\mathcal K} = A^{-\ast}_{\mathcal K}$.
	\end{enumerate}
\end{fact}

\section{The twin Alexander polynomial is mirror sensitive and 
can obstruct a knot from being classical}
\label{secnoin}

In this section, we prove Theorem~\ref{mthhayaku} and our claims in Remark~\ref{remiroiro} from the introduction. 
We have the following result, which is an immediate consequence
of the definition of Seifert matrices.

\begin{proposition}\label{prokihon}
Let $X$ be a Seifert matrix associated with 
a Seifert surface $V$ for 
a null-homologous link $L$ in a thickened surface.
Then the following hold:
\begin{enumerate}
\item $X^T$ is a Seifert matrix associated with 
the Seifert surface 
$-V$ for the link $-L$. 
\item $-X^T$ is a Seifert matrix associated with 
the Seifert surface 
$V^\ast$ for the link $L^\ast$. 
\item $-X$ is a Seifert matrix associated with 
the Seifert surface 
$-V^\ast$ for the link $-L^\ast$. 
\end{enumerate}
\end{proposition}

\begin{figure}
\begingroup%
  \makeatletter%
  \providecommand\color[2][]{%
    \errmessage{(Inkscape) Color is used for the text in Inkscape, but the package 'color.sty' is not loaded}%
    \renewcommand\color[2][]{}%
  }%
  \providecommand\transparent[1]{%
    \errmessage{(Inkscape) Transparency is used (non-zero) for the text in Inkscape, but the package 'transparent.sty' is not loaded}%
    \renewcommand\transparent[1]{}%
  }%
  \providecommand\rotatebox[2]{#2}%
  \newcommand*\fsize{\dimexpr\f@size pt\relax}%
  \newcommand*\lineheight[1]{\fontsize{\fsize}{#1\fsize}\selectfont}%
  \ifx\svgwidth\undefined%
    \setlength{\unitlength}{231.4962863bp}%
    \ifx\svgscale\undefined%
      \relax%
    \else%
      \setlength{\unitlength}{\unitlength * \real{\svgscale}}%
    \fi%
  \else%
    \setlength{\unitlength}{\svgwidth}%
  \fi%
  \global\let\svgwidth\undefined%
  \global\let\svgscale\undefined%
  \makeatother%
  \begin{picture}(1,1.04802612)%
    \lineheight{1}%
    \setlength\tabcolsep{0pt}%
    \put(0,0){\includegraphics[width=\unitlength,page=1]{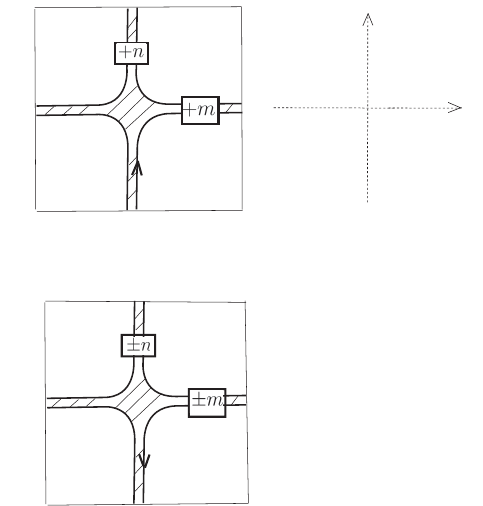}}%
    \put(-0.00091119,0.80640187){\makebox(0,0)[lt]{\lineheight{1.25}\smash{\begin{tabular}[t]{l}$\mathcal F$\end{tabular}}}}%
    \put(0.97183745,0.8150699){\makebox(0,0)[lt]{\lineheight{1.25}\smash{\begin{tabular}[t]{l}$x$\end{tabular}}}}%
    \put(0.75417281,1.0356238){\makebox(0,0)[lt]{\lineheight{1.25}\smash{\begin{tabular}[t]{l}$y$\end{tabular}}}}%
    \put(0,0){\includegraphics[width=\unitlength,page=2]{mn.pdf}}%
    \put(0.97183745,0.21017638){\makebox(0,0)[lt]{\lineheight{1.25}\smash{\begin{tabular}[t]{l}$x'$\end{tabular}}}}%
    \put(0.75417281,0.4368527){\makebox(0,0)[lt]{\lineheight{1.25}\smash{\begin{tabular}[t]{l}$y'$\end{tabular}}}}%
    \put(0.03679332,0.19831069){\makebox(0,0)[lt]{\lineheight{1.25}\smash{\begin{tabular}[t]{l}$\widehat{\mathcal{F}}$\end{tabular}}}}%
  \end{picture}%
\endgroup%

\caption{{The knots $K_{m,n}$ and $K^@_{m,n}$ with 
Seifert surfaces $V$ and $V^{@}$ and 
one-cycles $x$, $y$, $x'$, and $y'$ in the thickened torus.
The numbers in the boxes represent full twists.
}\label{figmn}}   
\end{figure}

\begin{example}\label{ex:non-classical}
Let $F = T^2$ and let $K_{m,n}$ be the knot in the thickened torus $\mathcal F := F\x [-1,1]$ shown in Figure~\ref{figmn}, where we represent the torus as a square $[-1,1] \times [-1,1]$ with the usual side identifications. See Figure~\ref{fignoin} for another reperesentation of the knot $K_{1,1}$. Note that $K_{m,n}$ is a classical knot when $mn = 0$. Indeed, if $m = 0$, one can destabilize to $S^2$ along the horizontal curve $[-1,1] \times \{0\}$, while we can destabilize to $S^2$ along the vertical curve $\{0\} \times [-1,1]$ when $n = 0$.

As before, consider $ F \x I = \mathcal F \cup \widehat{\mathcal F}$. 
Let $V\subset \mathcal F$ be the  Seifert surface for $K_{m,n}$
in Figure~\ref{figmn}. 
A basis of $H_1(V)$ consists of $x$ and $y$. 

Take $K^@_{m,n} \in \{-K_{m,n}, -{K^\ast_{m,n}} \}$ in $\widehat{\mathcal F}$ 
and the corresponding Seifert surface $V^{@}\subset\widehat{\mathcal F}$ drawn in the bottom of Figure~\ref{figmn}. 
A basis of $H_1(V^@)$ consists of $x'$ and $y'$. 

Construct a twin Seifert surface $V \# V^{@}$ from 
$V$ and  $V^{@}$ by a surgery using a vertical three-dimensional one-handle embedded in $F\x I$. 
We describe a Seifert matrix associated with $V\# V^@$. 
It is a $5 \x 5$ matrix $(\text{lk}(a_i,a_j))$,
where $a_1$ is a belt circle of the one-handle, 
$a_2=x$, $a_3=y$, $a_4=x'$, and $a_5=y'$. 

When $K^@_{m,n} = -K_{m,n}$, the Seifert matrix $X$ is  
\[
\begin{pmatrix}
0&0&0&0&0\\
0&m &\frac{1}{2}&0&\frac{1}{2} \\
0&-\frac{1}{2}& n&-\frac{1}{2}&0 \\
0&0&-\frac{1}{2}&m &-\frac{1}{2}\\
0&\frac{1}{2} &0&\frac{1}{2}&n\\
\end{pmatrix},
\] 
and we get the twin Alexander polynomial
\[
A^-_{K_{m,n}}(t) \doteq (mn(t-1)^2+t)^2.
\]
Note that we are taking the \emph{first} Alexander polynomial, 
which is the greatest common divisor of the $4 \times 4$ minors of $t  X - X^T$ in $\Z[1/2][t]$.

When $K^@_{m,n} = -K_{m,n}^\ast$, the Seifert matrix is
\[
\begin{pmatrix}
0&0&0&0&0\\
0&m &\frac{1}{2}&0&\frac{1}{2} \\
0&-\frac{1}{2}& n&-\frac{1}{2}&0 \\
0&0&-\frac{1}{2}&-m &-\frac{1}{2}\\
0&\frac{1}{2} &0&\frac{1}{2}&-n\\
\end{pmatrix}.
\] 
Hence, the twin Alexander polynomial is 
\[
\begin{split}
A^{-\ast}_{K_{m,n}}(t) &\doteq (mnt^2-2mnt+mn+1)(mnt^2-2mnt+mn+t^2) \\
&= m^2n^2(t-1)^4+mn(t^2+1)(t-1)^2+t^2.
\end{split}
\]
The polynomial $A^{-\ast}_{K_{m,n}}$ satisfies condition~\eqref{it:roots} of  Fact~\ref{facI} if and only if $mn=0$. 
Indeed, the roots of $A^{-\ast}_{K_{m,n}}(t)$ for $mn \neq 0$ and $mn \neq -1$ are
\[
1 \pm \mathrm{i}\sqrt{1/mn} \text{ \,\,\,and\,\,\, } \frac{mn \pm \mathrm{i} \sqrt{mn}}{mn + 1},
\]
so every root has multiplicity one. When $mn = 0$, the only root is $0$, which has multiplicity two. When $mn = -1$, the roots are $0$, $1/2$, and $2$, each of multiplicity one.
Therefore $K_{m,n}$ represents a non-classical virtual  knot if and only if $mn\neq0$. 

We have $A^{-}_{K_{m,n}} \neq A^{-\ast}_{K_{m,n}}$  if and only if  $mn \neq 0$. 
By part~\eqref{it:chiral} of Fact~\ref{facI}, 
the following are equivalent:
\begin{enumerate}
\item $K_{m,n}$ represents a non-classical virtual knot,
and $K_{m,n}$ and $K_{m,n}^\ast$ are not ambient isotopic in $T^2 \times [-1,1]$,
 \item $mn \neq 0$. 
\end{enumerate}
This implies Theorem~\ref{mthhayaku}. Note that we cannot conclude that $K_{m,n} \neq K_{m,n}^\ast$ as virtual knots. Indeed, $A^{-}_{K_{m,n}} \neq A^{-\ast}_{K_{m,n}}$ implies that $K_{m,n} \sqcup -K_{m,n}$ and $K_{m,n} \sqcup -K_{m,n}^*$ are not ambient isotopic in $T^2 \times [-1,1]$. It follows that $K_{m,n}$ and $K_{m,n}^*$ are not ambient isotopic in $T^2 \times [-1,1]$, but there might exist an automorphism of $T^2 \times [-1,1]$ that takes $K_{m,n}$ to $K_{m,n}^*$ that cannot be extended as the identity to the other copy of $T^2 \times [-1,1]$ containing $K_{m,n}$.
\end{example}

\begin{example}\label{ex:Kmm}
	Let $m$ be a positive integer. The mirror image of $K := K_{m,-m}$ is $K^* = K_{-m,m}$. By Theorem~\ref{mthhayaku}, we have $A^-_K \neq A^{-*}_K$ and that $K$ and $K^*$ are not isotopic. However, the diffeomorphism of $T^2 \times [-1,1]$ that swaps, preserving orientation, the meridian and longitude of $T^2$ maps $K$ to $K^*$, so $K$ and $K^*$ are diffeomorphic and hence also virtually equivalent.
\end{example}

\begin{figure}
\includegraphics[width=70mm]{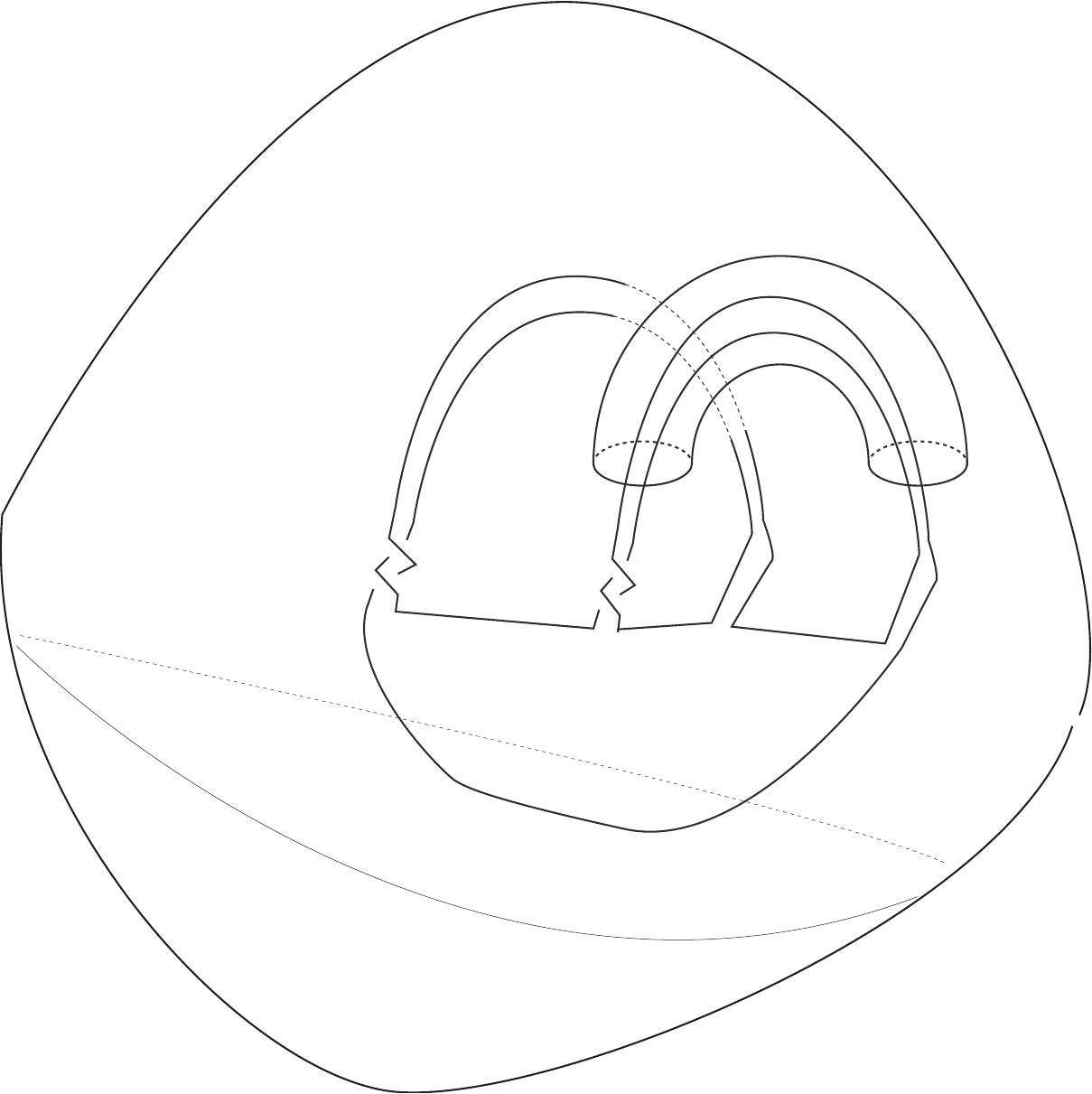}
\caption{{ The knot $K_{1,1}$ 
in the thickened torus. 
}\label{fignoin}}   
\end{figure}

\begin{example}\label{ex:signature}
The knots $K_{m,n}$ in Figure~\ref{figmn} are null-homologous in 
the thickened torus $T^2\x I$. 
However, the $d$-th twin Alexander polynomial is defined even if 
a knot is not null-homologous 
 in $T^2\x I$. 
We can use a 1-handle stabilization in virtual knot theory 
to make each knot $K_{m,n}$
into a knot in $F_2\x[-1,1]$ which is a non-vanishing cycle, 
where $F_2$ is a closed oriented surface of genus two. 

\begin{figure}
\includegraphics[width=90mm]{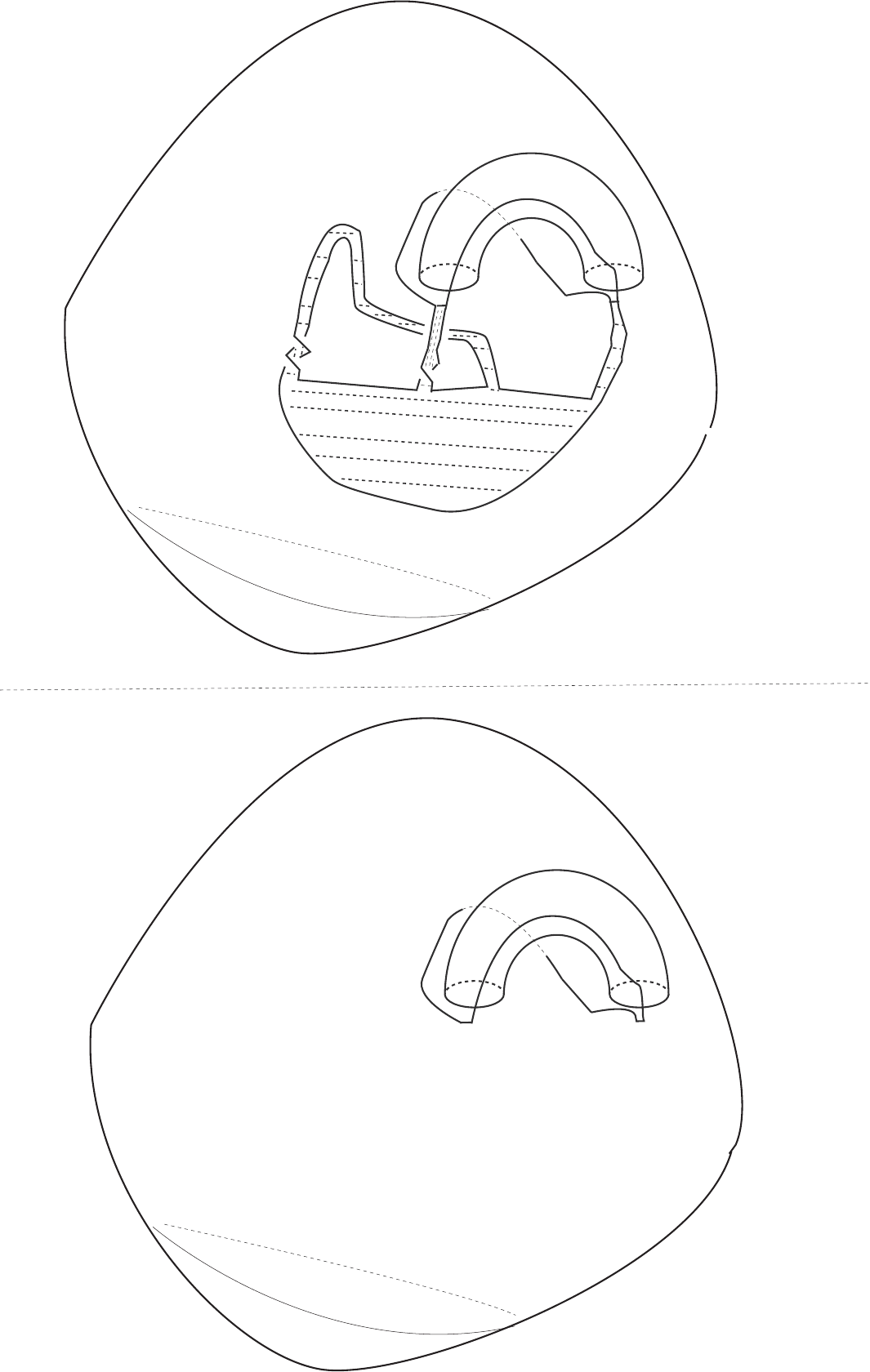}
\caption{{Top: A knot $K$  in the thickened torus and a compact oriented surface $E$ for $K$. Bottom: The knot $C$.}\label{figvt}}\label{fig:12}
\end{figure}

We now present another example in $T^2\x I$, where 
we apply the method of Section~\ref{subsec2waysII} for constructing the twin Seifert surface. 
Consider the knot $K$ in $T^2\x I$ and the compact oriented surface $E$ shown in the top of Figure~\ref{figvt}. Let $K^@ \in \{-K, -K^\ast\}$, and regard $T^2\x I$ as $\mathcal F\cup\widehat{\mathcal F}$ as above.  
Suppose that $K$  is in $\mathcal F$. 
Let $C \subset T^2 \x I$ be the knot shown in the bottom of Figure~\ref{figvt}. Assume that $C$ is in $F\x\{-1\} \subset \mathcal F$, and hence also in $\partial\widehat{\mathcal F}$. 
We construct a type zero Seifert surface $W$ for $K\sqcup K^@$
from $E$, $C \x [-1,0]$ or $C \times [0,1]$, and the corresponding surface $E^@$ for $K^@$. 

When $K^{@}=-K$, a Seifert matrix is 
\[
\begin{pmatrix}
0&0&0&0&0\\
0&1 &1&0&0 \\
0&0 & 1&0&0 \\
0&0&0&1 &0\\
0&0&0&1&1\\
\end{pmatrix},
\]
and hence the twin Alexander polynomial
\[
A^-_K \doteq (t^2-t+1)^2.
\]

When $K^{@}=-K^\ast$, a Seifert matrix is 
\[
\begin{pmatrix}
0&0&0&0&0\\
0&1 &1&0&0 \\
0&0 & 1&0&0 \\
0&0&0&-1 &-1\\
0&0&0&0&-1\\
\end{pmatrix}.  
\]
Hence, the twin Alexander polynomial 
\[
A^{-\ast}(K) \doteq (t^2-t+1)^2.
\]

The Jones polynomial of $K$ (recall the penultimate paragraph  of Section~\ref{subsecvk}) is not the Jones polynomial of any classical knot \cite{Kauffman, Kauffmani}, and
hence $K$ is non-classical.  
Using the above Seifert matrix associated with $K \sqcup -K$, 
the twin signature $\sigma^-(K) \neq 0$. Therefore, the non-classical knot $K$ is not isotopic to $K^\ast$.
\end{example}

\begin{question}\label{quemato} 
We conclude this section with a few questions. Let $K$ be a knot in a thickened surface.
\begin{enumerate}
\item Are the coefficients of the twin Alexander polynomial of $K$ always integers? 
\item Do all twin Seifert surfaces for $K \sqcup K^@$
have a Seifert matrix whose entries are integers?
\item Is $A^-_K$ always a square of a polynomial?  
\item Is $A^{-}_K$ or $A^{-\ast}_K$ mirror sensitive up to isotopy? 
\item Is the signature $\sigma^{-\ast}(K)$ always zero? 
\end{enumerate}
\end{question}

\begin{question}
	Let $K$ be a knot in a thickened surface $F\x[-1,1]$.  
	Are the twin Alexander polynomials $A^-_K$ 
	and $A^{-\ast}_K$ determined by 
	the colored Jones polynomials of $K\sqcup -K$ 
	and $K\sqcup -K^\ast$, respectively? 
	If $F=S^2$, the answer is affirmative by  
	Bar-Natan and Garoufalidis~\cite{BaGa}. Note that one can define the colored Jones polynomial for links in thickened surfaces and for virtual links. 
\end{question}

\section{The Alexander module associated with twin Seifert surfaces}\label{secmodu} 

Let $K$ be a knot in $F\x [-1,1]$ and let $K^@ \in \{-K, -K^\ast\}$.
Consider $K\sqcup K^{@}$ in $F\x I$ as above.
Let $N$ denote a tubular neighborhood of $K\sqcup K^{@}$ in $F\x I$. By construction, a twin Seifert surface for $K$ 
determines a unique homology class in
\[  
H_2(F\x [-1,1] \setminus \text{Int}(N), \partial N).
\] 
Therefore, the twin Seifert surfaces   for $K$
 induce a unique cyclic covering $\widetilde{X}$ of 
$F\x [-1,1] \setminus \text{Int}(N)$.  
Then $H_1(\widetilde{X};\Z)$ is
naturally a $\Z[t,t^{-1}]$-module, and is of the form
\[
\Z[t,t^{-1}]\oplus \dots \oplus\Z[t,t^{-1}]
\oplus(\Z[t,t^{-1}]/\lambda_1)\oplus \dots \oplus(\Z[t,t^{-1}]/\lambda_l)
\] 
for some polynomials $\lambda_1, \dots, \lambda_l$.

\begin{definition}
The polynomial 
$\lambda_1 \cdots \lambda_l$ 
is an isotopy invariant of $K$, and we call it 
the\emph{ cc twin polynomial} for $K$, where ``cc'' stand for ``cyclic covering.''
\end{definition}

If $K$ is a knot in the thickened sphere,  that is,  
a classical knot in $S^3$, 
the cc twin polynomial and the twin Alexander polynomial are equivalent. 
So it is natural to ask the following question. 

\begin{question}\label{quecov}
What is the relation between 
the cc twin polynomial and the twin Alexander polynomial? 
\end{question}

The Alexander polynomial of a classical knot $K$ in $S^3$ is 
Reidemeister torsion. Its refinement, Turaev torsion 
\cite{J, OSap, Turaev, Turaevronbun}, is the graded Euler characteristic of 
the knot Floer homology of $K$; see \cite{J, OSap}.  
So it is natural to ask the following question. 

\begin{question}\label{queFl}
Does a kind of Floer homology determine 
the twin Alexander polynomial
and  the  cc twin polynomial for knots in a thickened surface, and for virtual knots?
\end{question}

Although we do not answer Question~\ref{queFl} in this paper, 
we discuss a few relations between Floer homology 
and links in thickened surfaces and virtual links
in Section~\ref{sec:Floer}.

\section{Web Seifert surfaces and $\varepsilon(K)$ for knots in thickened surfaces}\label{sectypeII}  

In this section, we introduce another type of Seifert surface for knots in thickened surfaces, and define a numerical invariant using them.

\begin{definition}\label{def:web}
	Let $K$ be a knot in a thickened surface $F \times [-1,1]$.
	A \emph{web Seifert surface} for $K$ is a compact, connected, oriented surface $W$ with boundary, properly embedded in $F \times [-1,1]$, such that 
	\begin{enumerate}
		\item $\partial W \setminus (F \times \{-1,1\}) = K$,
		\item\label{it:ast} $W \cap (F \times \{-1,1\}) \neq \emptyset$.
	\end{enumerate}
\end{definition}

Clearly, every knot admits a web Seifert surface.
We will explain why we impose condition~\eqref{it:ast} in Remark~\ref{remdjd}. 

\begin{remark}\label{rematd}
Let $K$ be a knot in $\mathcal F = F\x[-1,1]$. 
Take $K^@ \in \{-K, -K^\ast\}$ in $\widehat{\mathcal F}$, and 
construct $F\x I$ from 
$\mathcal F$ and  $\widehat{\mathcal F}$
 as we did in Section~\ref{secwakatta}. Let $V$ be a twin Seifert surface for $K \sqcup K^@$.
Then $V \cap \mathcal F$ is a web Seifert surface for $K$. 
Not every web Seifert surface arises this way, since
$V \cap \mathcal F$ intersects only one connected component of 
$\partial \mathcal F$, while  a web Seifert surface might intersect both. 
 \end{remark}

\begin{definition}\label{defep}
Let $K$ be a knot in a thickened surface.	
We define $\varepsilon(K)$ as the maximal Euler characteristic of a web Seifert surface for $K$. 
If $\mathcal K$ is a virtual knot, we define $\varepsilon(\mathcal K)$ to be the $\varepsilon(K)$ for $K$ the minimal genus representative of $\mathcal K$. 
\end{definition}
 
\begin{remark}\label{rempp}
Let $K$ be the trivial knot in a thickened surface. 
 By property~\eqref{it:ast} of web Seifert surfaces, 
 $\varepsilon(K)=0$. 
 If we did not impose property~\eqref{it:ast}, we would have  $\varepsilon(K) = 1$. 
 This difference will be important later; see Remark~\ref{remdjd}.  

\begin{figure}
	\includegraphics[width=45mm]{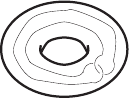}
	\caption{{A knot in a thickened torus.}\label{figchigau}}   
\end{figure}
  
 If $\varepsilon(K)=0$, then $K$ is virtually trivial.
 If $K$ represents the virtual knot $\mathcal K$, then we have $\varepsilon(K) \ge \varepsilon(\mathcal K)$. However 
 $\varepsilon(K)\neq\varepsilon(\mathcal K)$ in general; 
 for example in the case of the knot in Figure~\ref{figchigau}. 
 \end{remark}  
  
We will prove in Section~\ref{sec:Floer} that 
a type of Floer homology determines   $\varepsilon(K)$.

\section{Bracelet Seifert surfaces}\label{sectypeIII}

\subsection{Bracelet Seifert surfaces and $\xi(K)$ for knots in thickened surfaces}\label{subsecmist}

Let $K$ be a knot in a thickened surface $\mathcal F=F\x[-1,1]$. 
Take $-K^\ast$ in $\widehat{\mathcal F} = F \times [-1,1]$.   
Consider $F \x I = \mathcal F \cup \mathcal F_{-K^\ast}$ as in Section~\ref{secwakatta}. 
Construct $F\x S^1$ from $F\x I$ by identifying the two copies of $F$ in  the boundary of $F\x I$; 
see Figure~\ref{figFS2}. 
Then $K\sqcup -K^\ast$ is embedded in $F\x S^1$.  
(We do not consider $-K$ here.)

\begin{definition}
	For a knot $K$ in $F \x [-1,1]$, we define 
	a {\it bracelet Seifert surface} for $K \sqcup -K^\ast$
	to be  a compact, connected, oriented surface $S$ embedded in $F \times S^1$ with the following properties:
	\begin{enumerate}
		\item $\partial S = K\sqcup -K^\ast$. 
		\item Take a twin Seifert surface $W$ for 
		$K\sqcup -K^\ast$ in $F\x I$. 
		Then $W$ is embedded in $F \x S^1$ naturally. We require that 
		\[ 
		[S, K\sqcup -K^\ast] = [W, K\sqcup -K^\ast] \in H_2(F\x S^1, K\sqcup -K^\ast).
		\]  
	\end{enumerate}
\end{definition}

\begin{remark}\label{remte}
	Let $S$ be a bracelet Seifert surface for $K \sqcup -K^\ast$ in $F\x S^1$. By definition, $S$ is connected, hence $S \cap \mathcal{F}$ satisfies condition~\eqref{it:ast} of Definition~\ref{def:web}. So $S \cap\mathcal F$ is  a web Seifert surface  in $\mathcal F$  for $K$.
\end{remark}

\begin{definition}\label{defxi}
	Let $K$ be a knot in a thickened surface $F \times [-1,1]$.	
	Then $\xi(K)$ denotes
	the minimal genus of a bracelet Seifert surface for $K \sqcup -K^\ast$ in $F \x S^1$.
	If $\mathcal K$ is a virtual knot, we define $\xi(\mathcal K)$ to be $\xi(K)$ for $K$ the minimal genus representative of $\mathcal K$. 
\end{definition}

Recall $\varepsilon(K)$ from Definition~\ref{defep}.

\begin{theorem}\label{thmmuep}
	Let $K$ be a knot in $\mathcal F=F\x[-1,1]$. 
	Then 
	\[
	\xi(K)=-\varepsilon(K).
	\]  
\end{theorem}

\begin{proof}
	Let $S$ be a bracelet 
	Seifert surface for  $K \sqcup -K^\ast$ in $F \times S^1$. 
	Let 
	$V := S \cap \mathcal F$ and $W := S \cap \mathcal F_{-K^\ast}$. 
	Then  $V$ is a web Seifert surface for 
	$K$ and $W$ is a web Seifert surface for $-K^\ast$.  
	We have 
	\[
	-\frac{\chi(V)+\chi(W)}{2} = g(S),
	\] 
	since $S = V \cup W$ is a connected surface with two boundary components. 
	
	Without loss of generality, suppose that $\chi(V) \ge \chi(W)$.
	Put $-V^\ast$ in $\mathcal F_{-K^\ast}$.
	By the symmetric construction,  
	$V \cup -V^\ast$ is a bracelet Seifert surface 
	for $K \sqcup -K^\ast$ in $F\x S^1$. 
	As $\chi(V) = \chi(-V^\ast)$, we have
	\[
	-\frac{\chi(V) + \chi(-V^\ast)}{2} \le -\frac{\chi(V)+\chi(W)}{2}.
	\] 
	By construction, $V \cup -V^\ast$ is connected, hence
	\[
	-\frac{\chi(V) + \chi(-V^\ast)}{2} = \displaystyle g(V \cup -V^\ast).
	\]
	Therefore, $g(V \cup -V^\ast) \le g(S)$. 
	So there is a bracelet Seifert surface $S$ such that   
	$g(S) = \xi(K)$ and  $W = -V^\ast$.
	For such an $S$, we have
	\[
	g(S) = g(V \cup -V^\ast) = -\frac{\chi(V) +\chi(-V^\ast)}{2} = -\chi(V).
	\] 
	Hence $\xi(K) \ge -\varepsilon(K)$.
	
	Conversely, let $V$ be a web Seifert surface for $K$ such that $\chi(V) = \varepsilon(K)$. 
	Then $V \cup -V^\ast$ is a bracelet Seifert surface for $K \sqcup -K^\ast$ in $F \times S^1$ of genus
	\[
	-\frac{\chi(V )+ \chi(-V^\ast)}{2} = -\varepsilon(K).
	\] 
	Therefore $\xi(K) \le -\varepsilon(K)$.  
	This completes the proof of Theorem \ref{thmmuep}. 
\end{proof}

\begin{remark}\label{remdjd}
	We comment on what we announced in  Remark~\ref{remte}. 
	If we do not impose condition~\eqref{it:ast} in Definition~\ref{def:web}, the above doubling construction might not produce a bracelet Seifert surface as the result might not be connected. Furthermore, we would not have
	\[
	-\frac{\chi(V) + \chi(-V^\ast)}{2}=\displaystyle g(V \cup -V^\ast).
	\]  
\end{remark}

The quantities $\xi(K)$ and $\varepsilon(K)$ are 
related to a Floer homology; see Section~\ref{secshikashi}.

\subsection{The Alexander module associated with bracelet Seifert surfaces}\label{subsectildeo}

Let $K$ be a knot in a thickened surface $F\x[-1,1]$. 
A bracelet 
Seifert surface 
in $F\x S^1$ for $K \sqcup -K^\ast$ 
determines a unique class in 
\[
H_2 \left(F \x S^1 \setminus N(K\sqcup -K^\ast), \partial N(K\sqcup -K^\ast)\right).
\]  
Therefore, the bracelet Seifert surfaces 
induce a cyclic covering space $\widetilde{G}$ of 
$F\x S^1 \setminus N(K \sqcup -K^\ast)$.
Then $H_1(\widetilde{G};\Z)$ is naturally a $\Z[t,t^{-1}]$-module 
\[
\Z[t,t^{-1}]\oplus\cdots \oplus\Z[t,t^{-1}]
\oplus(\Z[t,t^{-1}]/\lambda_1)\oplus\cdots \oplus(\Z[t,t^{-1}]/\lambda_l).
\] 

\begin{definition}
The polynomial 
$\lambda_1\cdots \lambda_l$ 
is an isotopy invariant of $K\subset F\x[-1,1]$, and we call it 
the
{\it cc bracelet polynomial} of $K$. 
\end{definition}

From any twin Seifert surface in $F\x I$
one can obtain a bracelet Seifert surface in $F\x S^1$
by identifying the two boundary components.
It is a natural question to try to relate the cc bracelet polynomial 
to the cc twin polynomial; see Section~\ref{secmodu} and 
Remark \ref{remte}.   
Note that the former is defined by using 
a cyclic covering space of a closed manifold while 
the latter of a manifold with non-empty boundary.

\section{Floer homology} \label{sec:Floer}

\subsection{A review of Heegaard Floer homology}\label{secFloer} 

Heegaard Floer homology, defined by Ozsv\'ath and Szab\'o  \cite{OSmfd, OSap}, is a package of invariants for three-manifolds and four-manifolds. It also includes a knot invariant defined independently by Ozsv\'ath--Szab\'o~\cite{OSkf} and Rasmussen~\cite{Ras}, called knot Floer homology. It detects the unknot~\cite{OSg} and fibredness \cite{J2, Ni}. 

Sutured Floer homology, defined by the first author~\cite{J},
is a common extension of the hat version of Heegaard Floer homology of closed three-manifolds and knot Floer homology
to three-manifolds with boundary, more specifically, to sutured manifolds.  
Knot Floer homology is defined for null-homologous knots in closed oriented three-manifolds. 
On the other hand, 
sutured Floer homology is defined for arbitrary knots.

Bordered Floer homology, defined by Lipshitz, Ozsv\'ath, and Thurston~\cite{LOT},
is an invariant of 3-manifolds with parametrized boundary. 
Sutured Floer homology for different sutures and bordered Floer homology 
determine each other in a suitable sense. 

When we are given an invariant of links in thickened surfaces, 
one can ask whether it is changed by handle stabilization.
That is, is the invariant a virtual link invariant? 
In the following sections, we study 
whether we can extend Floer homology to virtual links.

\subsection{Floer homologies for knots in thickened  surfaces}\label{secSurf} 

In this section, we outline a number of constructions for defining 
Heegaard Floer homology for a knot in a thickened surface.

Let $K$ be a knot in a thickened surface $\mathcal F=F\x[-1,1]$.
Identify $F\x\{1\}$ with 
$F\x\{-1\}$. We obtain a knot $K$ in $F\x S^1$;
see Figure~\ref{figFS1}. 
 
Let $K^@ \in \{-K, -K^\ast\}$ in $\widehat{\mathcal F}$.   
Consider $F \x I = \mathcal F \cup\widehat{\mathcal F}$ as before. 
Make $F\x S^1$ from $F\x I$ 
by identifying the two copies of $F$ in  the boundary of $F\x I$.
Then we can view $K \sqcup K^@$ as a link in $F \times S^1$; 
see Figure~\ref{figFS2}. 

\begin{figure}
\begingroup%
  \makeatletter%
  \providecommand\color[2][]{%
    \errmessage{(Inkscape) Color is used for the text in Inkscape, but the package 'color.sty' is not loaded}%
    \renewcommand\color[2][]{}%
  }%
  \providecommand\transparent[1]{%
    \errmessage{(Inkscape) Transparency is used (non-zero) for the text in Inkscape, but the package 'transparent.sty' is not loaded}%
    \renewcommand\transparent[1]{}%
  }%
  \providecommand\rotatebox[2]{#2}%
  \newcommand*\fsize{\dimexpr\f@size pt\relax}%
  \newcommand*\lineheight[1]{\fontsize{\fsize}{#1\fsize}\selectfont}%
  \ifx\svgwidth\undefined%
    \setlength{\unitlength}{209.67323735bp}%
    \ifx\svgscale\undefined%
      \relax%
    \else%
      \setlength{\unitlength}{\unitlength * \real{\svgscale}}%
    \fi%
  \else%
    \setlength{\unitlength}{\svgwidth}%
  \fi%
  \global\let\svgwidth\undefined%
  \global\let\svgscale\undefined%
  \makeatother%
  \begin{picture}(1,0.76330992)%
    \lineheight{1}%
    \setlength\tabcolsep{0pt}%
    \put(0,0){\includegraphics[width=\unitlength,page=1]{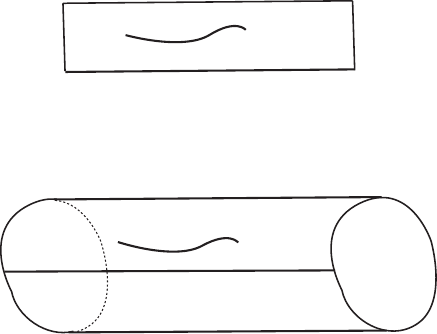}}%
    \put(0.59325024,0.67375943){\makebox(0,0)[lt]{\lineheight{1.25}\smash{\begin{tabular}[t]{l}$K$\end{tabular}}}}%
    \put(0.57560858,0.19187275){\makebox(0,0)[lt]{\lineheight{1.25}\smash{\begin{tabular}[t]{l}$K$\end{tabular}}}}%
    \put(0.84062457,0.67329255){\makebox(0,0)[lt]{\lineheight{1.25}\smash{\begin{tabular}[t]{l}$\mathcal F$\end{tabular}}}}%
    \put(0,0){\includegraphics[width=\unitlength,page=2]{FS1.pdf}}%
  \end{picture}%
\endgroup%

\caption{{ 
		The knot $K$ is in $F\x S^1$. 
}\label{figFS1}}   
\end{figure}

\begin{figure}
\begingroup%
  \makeatletter%
  \providecommand\color[2][]{%
    \errmessage{(Inkscape) Color is used for the text in Inkscape, but the package 'color.sty' is not loaded}%
    \renewcommand\color[2][]{}%
  }%
  \providecommand\transparent[1]{%
    \errmessage{(Inkscape) Transparency is used (non-zero) for the text in Inkscape, but the package 'transparent.sty' is not loaded}%
    \renewcommand\transparent[1]{}%
  }%
  \providecommand\rotatebox[2]{#2}%
  \newcommand*\fsize{\dimexpr\f@size pt\relax}%
  \newcommand*\lineheight[1]{\fontsize{\fsize}{#1\fsize}\selectfont}%
  \ifx\svgwidth\undefined%
    \setlength{\unitlength}{206.93085911bp}%
    \ifx\svgscale\undefined%
      \relax%
    \else%
      \setlength{\unitlength}{\unitlength * \real{\svgscale}}%
    \fi%
  \else%
    \setlength{\unitlength}{\svgwidth}%
  \fi%
  \global\let\svgwidth\undefined%
  \global\let\svgscale\undefined%
  \makeatother%
  \begin{picture}(1,0.90392603)%
    \lineheight{1}%
    \setlength\tabcolsep{0pt}%
    \put(0,0){\includegraphics[width=\unitlength,page=1]{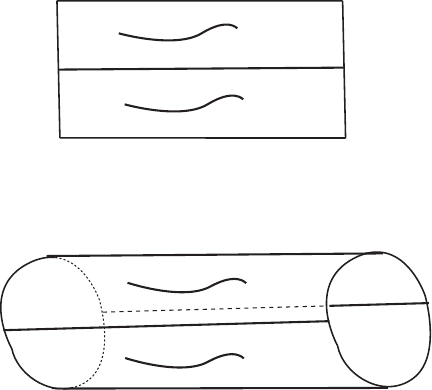}}%
    \put(0.8377709,0.8052704){\makebox(0,0)[lt]{\lineheight{1.25}\smash{\begin{tabular}[t]{l}$\mathcal F$\end{tabular}}}}%
    \put(0.83621245,0.6588022){\makebox(0,0)[lt]{\lineheight{1.25}\smash{\begin{tabular}[t]{l}$\widehat{\mathcal{F}}$\end{tabular}}}}%
    \put(0.56973605,0.82708515){\makebox(0,0)[lt]{\lineheight{1.25}\smash{\begin{tabular}[t]{l}$K$\end{tabular}}}}%
    \put(0.58220284,0.66347653){\makebox(0,0)[lt]{\lineheight{1.25}\smash{\begin{tabular}[t]{l}$K^@$\end{tabular}}}}%
    \put(0.59311117,0.23809494){\makebox(0,0)[lt]{\lineheight{1.25}\smash{\begin{tabular}[t]{l}$K$\end{tabular}}}}%
    \put(0.59155282,0.06357927){\makebox(0,0)[lt]{\lineheight{1.25}\smash{\begin{tabular}[t]{l}$K^@$\end{tabular}}}}%
    \put(0,0){\includegraphics[width=\unitlength,page=2]{FS2.pdf}}%
  \end{picture}%
\endgroup%

	\caption{{ 
			The link $K \sqcup K^{@}$ in  $F\x S^1$. 
		}\label{figFS2}}   
\end{figure}

Recall the linking number of a two-component link in a thickened surface from Section~\ref{subseclk}. 
In particular, even if $K$ is not null-homologous,  
we can define $r$-framed surgery along $K$ in $\mathcal F$ for $r \in\Q \cup \{\infty\}$.   
If we carry out the $r$-framed surgery on $\mathcal F$ along $K\subset \mathcal F$, 
then glue $F \times \{-1\}$ to $F \times \{1\}$, we obtain the three-manifold $M_{K,r}$. 

The three-manifold $M_{K\sqcup K^{@}, (r,r')}$ is defined as follows: 
Carry out $r$-framed surgery on $\mathcal F \subset F \times I$ 
along $K$, and
$r'$-framed surgery on $\widehat{\mathcal F} \subset F \times I$ along 
$K^{@}$.  
Then we glue the two copes of $F$ in the boundary of the surgered $F \times I$.

We can consider the Heegaard Floer homologies $\text{HF}^\circ(M_{K,r})$
and $\text{HF}^\circ(M_{K \sqcup K^@, (r, r')})$ for $r$, $r' \in \Q \cup \{\infty\}$
and $\circ \in \{\, \widehat{\,\,}, +, -, \infty \,\}$.

Let $L$ be a link in a closed, oriented three-maifold $Y$.
The \emph{sutured manifold $Y(L)$ complementary to $L$} is $Y \setminus N(L)$ 
with two oppositely oriented meridional sutures on each boundary component;
see \cite[Example 2.4]{J}.
Sutured Floer homology of $Y(L)$ is defined even when $L$ is 
not null-homologous; see Section~\ref{secFloer}. 
Hence, another way to define Floer homology for a knot in a thickened surface is as follows:

\begin{definition}
Let $K$ be a knot in the thickened surface $\mathcal F := F \times [-1,1]$. Then let	
\[
\begin{split}
\HFK(F \times S^1, K) &:= \SFH\left((F \times S^1)(K)\right), \\
\HFL(F \times S^1, K \sqcup K^@) &:= \SFH\left((F \times S^1)(K \sqcup K^@)\right).
\end{split}
\]
\end{definition}

An idea for using Floer homologies for links in thickened surfaces is suggested in the last section of
\cite{KauffmanOgasasq}. 
In this paper, we carry out that idea.

\subsection{Floer homology for virtual knots}\label{secideFHVnew}

We now introduce Heegaard Floer homology
for virtual knots.
We review a few basic results in virtual knot theory. 
Assume that a knot $K$ in a thickened surface $F\x [-1,1]$ represents a virtual knot $\mathcal K$.
Then we say that $(F, K)$ represents $\mathcal K$.

\begin{definition}\label{thmrepgen!}
The {\it representing genus} $m(\mathcal K)$ of the virtual knot $\mathcal K$
is 
\[
\min \{g : (F_g, K) \text{ represents } \mathcal K \},
\] 
where $F_g$ is the closed oriented surface of genus $g$.
\end{definition}

Kuperberg \cite{Kuperberg} proved the following theorem.

\begin{theorem}\label{thmKuperberg!} 
Let $\mathcal K$ be a virtual knot, and 
 $m$ the representing genus of $\mathcal K$. 
\begin{enumerate}
\item Suppose that two knots $K$ and $K'$ in $F_m\x[-1,1]$ represent $\mathcal K$. 
Then  $K$ is obtained from $K'$ by 
an orientation-preserving diffeomorphism of $F_m\x[-1,1]$, 
taking 
$F\x\{1\}$ to itself.  
\item \label{it:destab}
Assume that $(F, J)$ represents $\mathcal K$.  
Then we can obtain the minimal representing surface of $\mathcal K$ 
from $(F,J)$ only using destablizations.  
\end{enumerate}
\end{theorem}

Manturov \cite{Mvbunrui}\cite[\S2.4]{IMvbunrui}
generalized Kuperberg's theorem  
and proved the following result.

\begin{theorem}\label{thmManturov!}
In part~\eqref{it:destab} of Theorem~\ref{thmKuperberg!}, 
for any $F$, 
there is an algorithm  to construct 
the minimal representing surface of $\mathcal K$.  
\end{theorem}

We go back to our original purpose: to introduce Heegaard Floer homology for virtual links.
Let $\mathcal K$ be a virtual knot.
By Theorems~\ref{thmKuperberg!}
and~\ref{thmManturov!},  
we can construct  
a unique knot $K$ in the minimal genus surface $F$ 
which represents  $\mathcal K$,  
by a finite number of explicit operations. 
Define 
a Floer homology of $\mathcal K$ 
to be 
one of the Floer homologies of $K$ introduced in Section~\ref{secSurf}.  

\begin{remark} \label{rem:questions}
Suppose that a link in a thickened surface $F\x[-1,1]$ is given.
In many cases, 
we can detect whether $F$ is a minimal representing surface, 
without using Manturov's algorithm. 
For example, let $L$ be a  link in $T^2\x[-1,1]$ that represents a virtual link $\mathcal L$.  
If we know that $L$ is non-classical, then we can conclude that $m(\mathcal L) = 1$. 
The twin Alexander polynomial 
can be used in some cases; see Theorem~\ref{mthhayaku}.  
The Jones polynomial can also be used; see \cite{Kauffman, Kauffmani} and Section~\ref{secnoin}.  
If the link has at least two components, we can sometimes use the linking number. E.g., the virtual links in Figure~\ref{figVHopf} 
have non-integer linking numbers, and are hence non-classical.

Our Floer homologies tell us about knots embedded in specific thickened surfaces. Virtual knot theory can be regarded as a theory of knots in thickened surfaces up to stabilization and destabilization, and by Kuperberg's theorem, we can concentrate on knots that are in the minimal genus representing surface. In this case, the knot type in the given minimal genus surface represents the type of the virtual knot and no stabilization is required. It is an interesting question whether there are Floer-homological quantities that are invariant under stabilization.
\end{remark}

\subsection{Stabilization and Floer homology}\label{secshikashi}

One-handle stabilization in virtual knot theory 
changes our Floer homologies in general. In this section, 
we show some examples.

Let $K$ be a knot in the three-ball $B^3$ and $F_g$ a closed, connected, oriented surface of genus $g$. 
Embed $B^3$ in $F_g \x[-1,1]$, 
and call the resulting knot $K_g$.  
Then the $K_g$ are all virtually equivalent. 

By identifying $F_g \times \{0\}$ and $F_g \times \{1\}$, we obtain a knot  $K_g$ in $F_g \x S^1$; see Figure~\ref{figFS1}.
Consider the sutured manifold $(F_g \times S^1)(K_g)$ complementary to $K_g$, where we take two meridional sutures on the boundary of the knot exterior. Recall that
\[
\HFK(F_g \times S^1, K) := \SFH\left((F_g \x S^1)(K_g)\right).
\] 

Embed $B^3$ in $S^3$, and consider the image of 
$K$ in $S^3$, which we also denote by $K$.
By \cite[Proposition 9.15]{J}, 
\[
\HFK(F_g \times S^1, K_g) \cong \HFK(K) \otimes \widehat{\HF}(F_g \x S^1)
\]
as abelian groups. Since we have 
$\widehat{\HF}(S^2 \x S^1)\cong\Z^2$ and $\widehat{\HF}(T^3)\cong\Z^6$, 
\[
\HFK(S^2 \times S^1, K_0) \not\cong \HFK(T^2 \times S^1, K_1).
\]
Hence, we have obtained the following:

\begin{claim}\label{claJFH}  
Although the above $K_0$ and $K_1$ 
represent the same virtual knot, 
their Floer homologies defined above differ.
\end{claim}

Let $K$ be a knot in a thickened surface $F \times [-1,1]$. 
The invariant $\xi(K)$ (and hence  $\varepsilon(K) = -\xi(K)$ by Theorem~\ref{thmmuep}) is 
determined by $\HFL(F \times S^1, K \sqcup -K^\ast)$:   
The homology class of a bracelet Seifert surface relative to the boundary 
is determined uniquely by definition.  
The Thurston norm of a class in $H_2((F \times S^1) \setminus N(K))$
is determined by $\HFL(F \times S^1, K \sqcup -K^*)$ according to
\cite{J3}.   

As stated above, $\varepsilon(K)$ is determined by Floer homology.
Some handle stabilizations in virtual knot theory change $\varepsilon(K)$ 
and hence $\HFL(F \times S^1, K \sqcup -K^\ast)$; see Remark~\ref{rempp}.

\subsection{Floer homology and the cc bracelet polynomial}

The cc bracelet polynomial can be recovered from the decategorification of sutured Floer homology~\cite{FJR}.
Indeed, let 
\[
(M, \gamma) := (F \times S^1)(K \sqcup -K^\ast)
\]
be the sutured manifold complementary to $K \sqcup -K^\ast$,
where $M := (F \times S^1) \setminus N(K \sqcup -K^\ast)$ and $\gamma$ consists of two oppositely oriented meridional sutures 
on each component of $\partial M$.
The relative Turaev torsion $\tau(M, \gamma)$ is defined in \cite{FJR}
as a $\text{Spin}^c$ refinement of the relative Reidemeister torsion of the pair $(M, R_-(\gamma))$.
A straightforward generalization of \cite[Lemma~6.3]{FJR} 
gives that $\tau(M, \gamma)$ is the Alexander polynomial $\Delta_{K \sqcup -K^\ast} \in \Z[H_1(M)]$. Furthermore, 
\[
\tau(M, \gamma, \fs) = \chi(\text{SFH}(M,\gamma,\fs))
\]
for every $\text{Spin}^c$ structure $\fs \in \text{Spin}^c(M, \gamma)$.
Let $b \in H_2(M, \partial M)$ be the class of a bracelet Seifert surface, and consider the map 
\[
p \colon H_1(M) \to \Z
\]
given by intersection with $b$.
Then the cc bracelet polynomial of $K \sqcup K^\ast$
is $p(\Delta_{K \sqcup -K^\ast})$.

\section{The Behrens--Golla $\delta$-invariant and virtual knots}\label{sechirameki}

Let $K$ be a knot in a thickened surface $\mathcal F=F\x[-1,1]$, and
let $K^@ \in \{-K, -K^\ast\}$, as in Section~\ref{secwakatta}.   
Construct $F\x S^1$ from $\mathcal F$ containing $K$ and $\widehat{\mathcal F} = F \times [-1,1]$ containing $K^@$, as before. 
Then $F\x S^1$ includes the knot $K\sqcup K^@$; see Figure~\ref{figFS2}.
By attaching a four-dimensional one-handle~$h$ to $F \times S^1$, 
we obtain the three-manifold
\[
M^g := (F\x S^1) \# (S^1\x S^2),
\]
where $g$ is the genus of $F$.  
Inside $M^g$, attach a two-dimensional one-handle to $K \sqcup K^@$ contained in~$h$. We obtain the knot $A = K \# K^@$ in $M^g$; see Figure~\ref{fighirameita}.
Note that $K \# K^@$ is null-homologous in $M^g$.  
Figure~\ref{fighirameki} shows a Kirby diagram representing $M^g$, together with the knot $A$.
The linking number of $A$ and each component of the framed link 
in the diagram is zero. 
Let $M^g_r(A)$ be the closed oriented three-manifold 
obtained from $M^g$  
by a surgery along $A$ with framing $r \in \Z$.

\begin{figure}
\begingroup%
  \makeatletter%
  \providecommand\color[2][]{%
    \errmessage{(Inkscape) Color is used for the text in Inkscape, but the package 'color.sty' is not loaded}%
    \renewcommand\color[2][]{}%
  }%
  \providecommand\transparent[1]{%
    \errmessage{(Inkscape) Transparency is used (non-zero) for the text in Inkscape, but the package 'transparent.sty' is not loaded}%
    \renewcommand\transparent[1]{}%
  }%
  \providecommand\rotatebox[2]{#2}%
  \newcommand*\fsize{\dimexpr\f@size pt\relax}%
  \newcommand*\lineheight[1]{\fontsize{\fsize}{#1\fsize}\selectfont}%
  \ifx\svgwidth\undefined%
    \setlength{\unitlength}{141.8795162bp}%
    \ifx\svgscale\undefined%
      \relax%
    \else%
      \setlength{\unitlength}{\unitlength * \real{\svgscale}}%
    \fi%
  \else%
    \setlength{\unitlength}{\svgwidth}%
  \fi%
  \global\let\svgwidth\undefined%
  \global\let\svgscale\undefined%
  \makeatother%
  \begin{picture}(1,0.65048058)%
    \lineheight{1}%
    \setlength\tabcolsep{0pt}%
    \put(0,0){\includegraphics[width=\unitlength,page=1]{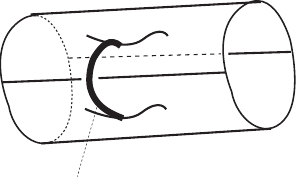}}%
    \put(0.07260371,0.00072447){\makebox(0,0)[lt]{\lineheight{1.25}\smash{\begin{tabular}[t]{l}A 4-dimensional 1-handle\end{tabular}}}}%
    \put(0.59432319,0.51832394){\makebox(0,0)[lt]{\lineheight{1.25}\smash{\begin{tabular}[t]{l}$K$\end{tabular}}}}%
    \put(0.59432319,0.26941724){\makebox(0,0)[lt]{\lineheight{1.25}\smash{\begin{tabular}[t]{l}$K^@$\end{tabular}}}}%
  \end{picture}%
\endgroup%

\caption{{
A four-dimensional one-handle surgery and a connected sum of knots change 
$K\sqcup K^{@}$ in $F\x S^1$ 
into 
$A=K\# K^{@}$ in $(F\x S^1)\# (S^1\x S^2)$. 
}\label{fighirameita}}   
\end{figure}
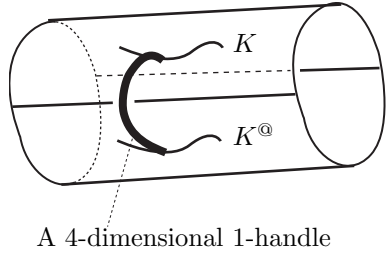

\begin{figure}
\begingroup%
  \makeatletter%
  \providecommand\color[2][]{%
    \errmessage{(Inkscape) Color is used for the text in Inkscape, but the package 'color.sty' is not loaded}%
    \renewcommand\color[2][]{}%
  }%
  \providecommand\transparent[1]{%
    \errmessage{(Inkscape) Transparency is used (non-zero) for the text in Inkscape, but the package 'transparent.sty' is not loaded}%
    \renewcommand\transparent[1]{}%
  }%
  \providecommand\rotatebox[2]{#2}%
  \newcommand*\fsize{\dimexpr\f@size pt\relax}%
  \newcommand*\lineheight[1]{\fontsize{\fsize}{#1\fsize}\selectfont}%
  \ifx\svgwidth\undefined%
    \setlength{\unitlength}{276.35854768bp}%
    \ifx\svgscale\undefined%
      \relax%
    \else%
      \setlength{\unitlength}{\unitlength * \real{\svgscale}}%
    \fi%
  \else%
    \setlength{\unitlength}{\svgwidth}%
  \fi%
  \global\let\svgwidth\undefined%
  \global\let\svgscale\undefined%
  \makeatother%
  \begin{picture}(1,0.84641426)%
    \lineheight{1}%
    \setlength\tabcolsep{0pt}%
    \put(0,0){\includegraphics[width=\unitlength,page=1]{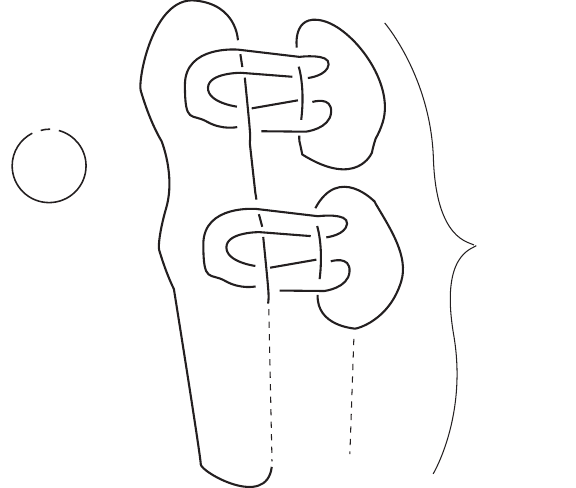}}%
    \put(0.85230584,0.40748977){\color[rgb]{0.1372549,0.12156863,0.1254902}\makebox(0,0)[lt]{\lineheight{1.25}\smash{\begin{tabular}[t]{l}g copies\end{tabular}}}}%
    \put(0,0){\includegraphics[width=\unitlength,page=2]{Kirby.pdf}}%
  \end{picture}%
\endgroup%

\caption{{ 
A Kirby diagram with a knot which represents 
$K\# K^{@}$ in $M^g$.
All components of the framed link have framing zero, which we omit.  
}\label{fighirameki}}   
\end{figure}

Suppose that $K_1$ in $F_1 \times [-1,1]$ is obtained from $K$ in $F \times [-1,1]$ by handle stabilizations in virtual knot theory,
and let $s$ be the genus of $F_1$. 
Construct the link $A_1$ in the three-manifold $M^s$ and the surgery $M^s_r(A_1)$ from $K_1$ and $F_1$,  as above. 
We can obtain $(M^s, A_1)$ from  $(M^g, A)$    
by performing four-dimensional two-handle surgeries, which correspond to the one-handle surgeries in virtual knot theory. 
The framed link representation after the surgery is obtained 
by increasing $g$ in Figure~\ref{fighirameki} to $s$. 

Thus, we obtain a four-dimensional cobordism
$W$ from $M^g_r(A)$ to $M^s_r(A_1)$.    
Then $H_2(W)$ has zero intersection form.  
Indeed, $W$ is obtained by attaching four-dimensional two-handles $h_1,...,h_\nu$ to $M_r^g(A) \x [0,1]$, where $\nu = 2(s - g)$. 
The linking number of each pair of components of the framed link in Figure~\ref{fighirameki} is zero.  Furthermore, all framings are zero by construction. 

Let $\delta$ be the Behrens--Golla invariant 
defined above Theorem~4.1 in \cite{BG}.  
Let $\fs$ and $\fs_1$ be $\text{Spin}^c$ structures on $M^g_r(A)$  and $M^s_r(A_1)$, respectively, with $c_1(\fs) = 0$ and $c_1(\fs_1) = 0$.
Then there is a $\text{Spin}^c$ structure $\ft$ on the cobordism $W$ such that $c_1(\ft) = 0$ and the restriction of $\ft$ to 
$M^g_r(A)$ is $\fs$ and to $M^s_r(A_1)$ it is $\fs_1$.  Then we have the following:

\begin{theorem}\label{thmwo!} 
Using the notation above, we have 
 \[
 \delta(M^g_r(A), \fs) \le \delta(M^s_r(A_1), \fs_1).
 \] 
\end{theorem}   

\begin{proof}
	This is an immediate consequence of \cite[Theorem~4.1]{BG}.
\end{proof}

\begin{corollary}\label{thmv}
Let $P$ be a knot in $F \x [-1,1]$
and  let $Q$ be a knot in $F_1 \x [-1,1]$. 
Construct $M^g_r(A)$ and $M^s_r(A_1)$, 
as above, where $A = P \# P^@$, $A_1 = Q \# Q^@$,
$g = g(F)$, $s = g(F_1)$,
and $r \in \Z$. If $s > g$ but
\[
\delta(M^g_r(A), \fs) > \delta(M^s_r(A_1),\fs_1)
\]
for $\text{Spin}^c$ structures $\fs$ and $\fs_1$ with vanishing first Chern classes,
then $(F_1, Q)$ cannot be obtained from $(F, P)$ using virtual one-handle
stabilizations. 
\end{corollary}  

In particular, by Theorem~\ref{thmKuperberg!}, if $(F, P)$ is minimal genus (e.g., a classical knot or the virtual trefoil), then we can conclude that $(F, P)$ and $(F_1, Q)$ are not virtually equivalent
whenever
\[
\delta(M^g_r(A), \fs) > \delta(M^s_r(A_1),\fs_1).
\]
We would like to have a specific example of this phenomenon.

\end{document}